\documentclass{article}
\usepackage{amsthm, amssymb, latexsym, amsmath}
\usepackage{graphicx,color}
\usepackage[usenames,dvipsnames,svgnames,table]{xcolor}
\newtheorem{theorem}{Theorem}[section]
\newtheorem{lemma}[theorem]{Lemma}
\newtheorem{proposition}[theorem]{Proposition}

\newtheorem{remark}{Remark}

\newtheorem{definition}{Definition}
\newcommand{\nc}{\newcommand}
\def\R{\mathbb{R}}
\def\C{\mathbb{C}}
\nc{\diff}[2]{\frac{d #1}{d #2}}
\nc{\diffn}[3]{\frac{d^{#3} #1}{d {#2}^{#3}}}
\nc{\pdiff}[2]{\frac{\partial #1}{\partial #2}}
\nc{\pdiffn}[3]{\frac{\partial^{#3} #1}{\partial{#2}^{#3}}}
\nc{\abs}[1] {\lvert #1 \rvert}
\nc{\norm}[2] {{\lVert #1 \rVert}_{#2}}
\nc{\threeline}[1] {\lvert \lvert \lvert #1 \rvert\rvert\rvert}
\nc{\iamMc}{\frac{i\alpha-M}{c}}
\nc{\iapMc}{\frac{i\alpha+M}{c}}

\nc{\bXYZ}{{\bf XYZ\ }}
\nc{\GW}{{\bf GW1\ }}
\nc{\Rem}{{\rm Rem}}
\nc{\loc}{{\rm loc}}
\nc{\cF}{{\cal F}}
\nc{\cG}{{\cal G}}
\nc{\cO}{{\cal O}}
\nc{\cQ}{{\cal Q}}
\nc{\cR}{{\cal R}}
\nc{\cS}{{\cal S}}
\nc{\sqrtE}{\mu}
\nc{\cI}{{\cal I}}
\nc{\cK}{{\cal K}}
\nc{\cL}{{\cal L}}
\nc{\cM}{{\cal M}}
\nc{\cN}{{\cal N}}
\nc{\cE}{{\cal E}}
\nc{\cH}{{\cal H}}
\nc{\cZ}{{\cal Z}}
\nc{\cT}{{\cal T}}
\nc{\rhoo}{{(z\cdot\xi)}}
\nc{\omegaa}{{(z\cdot\eta)}}

\nc{\order}{{\cal O}}
\nc{\ores}{{\omega_{\rm res}}}
\nc{\nit}{\noindent}
\nc{\Eplus}{E_+}
\nc{\Eminus}{E_-}
\nc{\Epm}{E_\pm}

\nc{\w}{\omega}
\nc{\eps}{\epsilon}
\nc{\e}{\varepsilon}
\nc{\g}{\gamma}
\nc{\z}{\zeta}
\nc{\G}{\Gamma}

\nc{\nn}{\nonumber}
\nc{\D}{\partial}
\nc{\pZ}{\partial_Z}
\nc{\pT}{\partial_T}
\nc{\pz}{\partial_z}
\nc{\pt}{\partial_t}
\nc{\vu}{\Vec u}
\nc{\vE}{\Vec {\cal E}}
\nc{\vr}{\Vec r}
\nc{\vrho}{\Vec \rho}
\nc{\Reps}{R^{\e}}
\nc{\Vreps}{\Vec \Reps}
\nc{\half}{\frac{1}{2}}
\nc{\bphi}{\bar{\phi}}
\nc{\efour}{{\Hat e}_4}
\nc{\marginnote}[1] {\marginpar{\tiny #1}}

\normalsize
\begin{document}
\title{A Resonance Problem in Relaxation of Ground States of Nonlinear Schr\"odinger Equations}\author{Zhou Gang\footnote{Partially supported by NSF grant DMS 1308985, 1443225} \
} \maketitle
\centerline{\small{ Department of Mathematics, California Institute of Technology, Mail Code 253-37,
Pasadena, California, 91125 U.S.A.}}
\setlength{\leftmargin}{.1in}
\normalsize \vskip.1in \setcounter{page}{1}
\setlength{\leftmargin}{-.2in} \setlength{\rightmargin}{-.2in}
\section*{Abstract}
In this paper we consider a resonance problem, in a generic regime, in the consideration of relaxation of ground states of semilinear Schr\"odinger equations. Different from previous results, our consideration includes the presence of resonance, resulted by overlaps of frequencies of different states. All the known key results, proved under non-resonance conditions, have been recovered uniformly. These are achieved by better understandings of normal form transformation and Fermi Golden rule. Especially, we find that if certain denominators are  zeros (or small), resulted by the presence of resonances (or close to it), then  cancellations between terms make the corresponding numerators proportionally small.  \vfil\eject
\tableofcontents
\section{Introduction}
We consider the following 3-dimensional semilinear Schr\"odinger equations
\begin{align}
i\partial_{t}\psi(x,t)=&-\Delta\psi(x,t)+V(x)\psi(x,t)+|\psi(x,t)|^{2}\psi(x,t), \label{eq:NLS}\\
\psi(x,0)=&\psi_{0}(x)\in \mathcal{H}^2(\mathbb{R}^3)\ \text{small}.\nonumber
\end{align}
where $V:\mathbb{R}^{3}\rightarrow \mathbb{R}$ is
the external potential.

Such equations arise in the theory of Bose-Einstein condensation,
nonlinear optics, theory of water waves and in other areas.

We start with formulating the problem. The potential $V:\ \mathbb{R}^{3}\rightarrow \mathbb{R}$ is a smooth, and rapidly decaying function. And if it is trapping potential, namely 
\begin{align}
\inf_{\|f\|_2=1}\langle f,\ (-\Delta+V)f\rangle=-e_0<0,
\end{align} then this linear unbounded self-adjoint operator $-\Delta+V$, mapping $\mathcal{L}^2$ into $\mathcal{L}^2$,  
has a ground state $\phi\in \mathcal{L}^2$ with eigenvalue $-e_0$. Moreover the eigenvalue must be simple. Besides the ground states the linear operator might have some other finitely many neutral modes with nonpositive eigenvalues $-e_k ,\ k=1,\cdots, N$.
Its continuous spectrum spans the interval $[0,\ \infty).$
It is well known that for the type of potential $V$ we chose, there is no positive eigenvalues, see e.g. \cite{RSIII}.

In the nonlinear setting, the ground state bifurcates into a family of solitary wave solutions, see e.g. \cite{TsaiYau02},
\begin{align}
\psi(x,t)=e^{i\lambda t}\phi^{\lambda}(x)
\end{align} with $\lambda\in \mathbb{R}$ being close to $e_0$ and $\phi^{\lambda}=C \sqrt{|e_0-\lambda|}\phi+O(|e_0-\lambda|^{\frac{3}{2}}).$

There is a rich literature on studying the orbital stability and
asymptotic stability of the soliton manifold. By results in
\cite{Wein1985, Wei86,GSS87} it is well known that that the ground state manifold is orbital stable in the
$\mathcal{H}^{1}$ space. After these, many attempts were made on
proving the asymptotic stability of the ground state manifold, see e.g. \cite{BP1, RoWe,TsaiYau02, MR1992875, BuSu, GS07, Cuccagna:03, SW-PRL:05, G1, NakanTsai12, ZwHo07}.

In \cite{BP1, RoWe,BuSu,TsaiYau02, MR1992875, SW-PRL:05}, it is assumed that the linear operator $-\Delta+V$ has a ground state (with eigenvalue $-e_0<0$), and only one simple neutral mode with eigenvalue $-e_1<0$ satisfying $2e_1<e_0.$
In \cite{Tsai2003}, multiple neutral modes was considered. Their eigenvalues $-e_k,\ k=1,\cdots,N$ must satisfy two conditions: (1) $2e_k< e_0$, and (2) the so-called non-resonance condition, namely there do not exist $n_k\in \mathbb{Z}, \ k=0,1,\cdots, N$ such that
\begin{align}
\sum |n_k|\not= 0\ \text{and}\ \sum_{k} n_k e_k=0,\label{eqn:nonRe1}
\end{align} see also the non-resonance conditions for multiple neutral modes in \cite{Cucca08, NakanTsai12}.

On the technical level, the condition \eqref{eqn:nonRe1} was needed to prevent small denominator from appearing.

In \cite{GaWe, GaWe2011}, the author, together with M. Weinstein, improved the above results by studying degenerate neutral modes, i.e.
\begin{align}
e_k=e_1, k=1,2,\cdots,  N,\label{eqn:nonRe2}
\end{align} or nearly degenerate.

The main purpose of the present paper is to include the presence of the resonance, specifically
by removing the conditions \eqref{eqn:nonRe1} and \eqref{eqn:nonRe2}, and to show all the proved results still hold, uniformly.

A graphic illustration is in Figure \ref{abc}.
\begin{figure}[!htb]
\centering
\includegraphics[width=11.5cm, height=2.5cm]{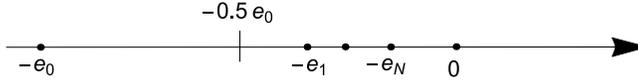}
\caption{spectrum of $-\Delta+V$}
\label{abc}
\end{figure}

On the technical level, we achieve this by re-defining normal form transformation and Fermi Golden rule. Especially we show that if some denominators are small, caused by the resonances, then their corresponding numerators are proportionally small, resulted by cancellations between terms. For the details, we refer to Sections \ref{sec:FGR} and \ref{NormalForms}, and Proposition \ref{prop:smallDivisor}.

To the best of our knowledge, our result and techniques are new.

{\bf{Related problems.}} The main motivation of the present work is to understand certain types of small divisor problems. In many literatures, for example in \cite{BourWang2008, Bene1988, Bovier86, BourgWang04}, in studying problems arising from statistical mechanics and PDE, the small divisor was avoided, by choosing initial conditions, for example.

The technical advantage in studying the present problem is that we only need to expand the solutions, in certain small parameters, finitely many times (two times), instead of infinitely many times as in \cite{BourWang2008}. We expect the normal form transformation invented here, which makes it easy to see the crucial cancellations between terms, together with algebraic structures observed in \cite{GS07, G1, Cucca08}, for higher order iterations, can be applied to the other problems. These will be addressed in subsequent papers.

{\bf{Structure of the paper.}}
The paper is organized as follows. Some basic properties of the equation are studied in Section \ref{HaGWP}. The linearized operator $L(\lambda)$, obtained by linearizing the ground state solution, is studied in Section \ref{sec:OperL}. The Fermi Golden rule condition is in Section \ref{sec:FGR}. The main Theorem is stated in Section \ref{MainTHM}. The decomposition of the solutions, and the governing equations for various parameters and functions are in Section \ref{SEC:effective}. The main Theorem is reformulated into Theorem \ref{GOLD:maintheorem} in Section \ref{sec:refor}. The proof of different parts of Theorem \ref{GOLD:maintheorem} are in Sections \ref{sec:tildeR}, \ref{NormalForms} and \ref{sec:NFT}. The main theorem is proved in Section \ref{ProveMain}.

Most parts of the paper follow the steps in the previous papers \cite{GaWe, GaWe2011}.  Technically the main differences are in normal form transformation in Section \ref{NormalForms} and subsequent sections.

\section*{Acknowledgments} The author wishes to thank Avy Soffer and Wei-Min Wang for a discussion on small divisor problems.

The author is partly supported by NSF Grant DMS 1308985, 1443225. 

\newpage

\section{Notation}\label{notation}
\begin{itemize}
\item[(1)]\ 
$
\alpha_+ = \max\{\alpha,0\},\ \ [\tau]=\max_{\tilde\tau\in Z}\ \{\tilde\tau\le\tau\}
$
\item[(2)]\ $\Re z$ = real part of $z$,\ \ $\Im z$ = imaginary part of $z$
\item[(3)] Multi-indices
\begin{align}
z&=(z_1,\dots, z_N)\ \in C^N,\ \bar{z}=(\overline{z_1},\dots, \overline{z_N})\\
&a\in \mathbb{N}^{N},\ z^a=z_1^{a_1}\cdot\cdot\cdot z_N^{a_N}\nonumber\\
|a|\ &=\ |a_1|\ +\ \dots\ +\ |a_N|
\nonumber
\end{align}

\item[(4)]\ $\mathcal{H}^s$\ =\ Sobolev space of order $s$
\item[(5)]  \begin{equation}
J\ =\ \left(\begin{array}{cc} 0 & 1\\ -1 & 0\end{array}\right),
\ \ H\ =\ \left(\begin{array}{cc} L_+ & 0\\ 0 & L_-\end{array}\right),\ \ 
 L=JH=\left(\begin{array}{cc} 0 & L_-\\ -L_+ & 0\end{array}\right)
\nonumber\end{equation}
\item[(6)] $\sigma_{ess}(L)=\sigma_c(L)$ is the essential (continuous) spectrum of $L$,\\ $\sigma_d=\C - \sigma_c(L)$ is the discrete spectrum of $L$.
\item[(7)]\ $P_d(L)$ bi-orthogonal projection onto the discrete spectral part of $L$
\item[(8)]\ $P_c(L)=I-P_d(L)$, bi-projection onto the continuous spectral part of $L$
\item[(9)]\ $\langle f,g\rangle = \int\ f\ \bar{g} $
\item[(10)]\ $\| f\|_p=$ $L^p$ norm,\ \ $1\le p\le\infty$.
\end{itemize}
\vfil\eject 
\section{Basic Properties}\label{HaGWP}
Equation \eqref{eq:NLS} is a Hamiltonian system on Sobolev space
$\mathcal{H}^{1}(\mathbb{R}^{3},\mathbb{C})$ viewed as a real space
$\mathcal{H}^{1}(\mathbb{R}^{3},\mathbb{R})\oplus
\mathcal{H}^{1}(\mathbb{R}^{3},\mathbb{R})$. The
Hamiltonian functional is: $$\mathcal{E}(\psi):=\int
[\frac{1}{2}(|\nabla\psi|^{2}+V|\psi|^{2})+\frac{1}{4}|\psi|^{4}].$$ 

Equation \eqref{eq:NLS} has the time-translational and gauge
symmetries which imply the following conservation laws: for any
$t\geq 0,$ we have
\begin{enumerate}
 \item[(CE)] Conservation of energy:\ \ \ 
  $\mathcal{E}(\psi(t))=\mathcal{E}(\psi(0));$
 \item[(CP)]
 Conservation of particle number:\\
  $\mathcal{N}(\psi(t))=\mathcal{N}(\psi(0)),$ where\ \ \ \
   $\mathcal{N}(\psi):=\int
 |\psi|^{2}.$
\end{enumerate}

In what follows we review the results of the existence of
soliton and their properties.

The following arguments are almost identical
to those in \cite{RoWe,BuSu,TsaiYau02} except that here we have multiple neutral modes (or
excited states), hence we state the results
without proof. We assume that the linear operator $-\Delta+V$ has
the following properties
\begin{enumerate}
\item[(NL)]
The linear operator $-\Delta+V$ has eigenvalues $-e_{0}<-e_{1}\leq \cdots\leq -e_{N}$
satisfying $e_{0}<2e_{1}$. $-e_{0}$ is the lowest eigenvalue with
ground state $\phi>0$, the eigenvalue $-e_{1},\ \cdots,\ -e_{N}$ might be degenerate with eigenvectors
$\xi_{1}^{lin},\xi_{2}^{lin},\cdot\cdot\cdot,\xi_{N}^{lin}.$
\end{enumerate}

In the nonlinear setting the ground state bifurcates into a family of solitary wave solutions of \eqref{eq:NLS}, see e.g. \cite{TsaiYau02},
\begin{align}
\psi(x,t)=e^{i\lambda t}\phi^{\lambda}(x)\label{eq:soli2}
\end{align} 
and the function $\phi^{\lambda}>0$ has the following properties, see e.g. \cite{TsaiYau02}.
\begin{lemma}\label{LM:groundNon}
Suppose that the linear operator $-\Delta+V$ satisfies the
conditions in [NL] above. Then there exists a constant
$\delta_{0}>0$ such that for any $\lambda \in [e_{0}-\delta,
e_{0})$ \eqref{eq:NLS} has solutions of the form
$\psi(x,t)=e^{i\lambda t}\phi^{\lambda}(x)\in \mathcal{L}^{2}$ with
\begin{align}\label{eq:phiAsy}
\phi^{\lambda}=\delta\phi+\cO(\delta^3)
\end{align} and
$\delta=(\int
\phi^{4}(x)dx)^{-\frac{1}{2}}(e_{0}-\lambda)^{1/2}+o((e_{0}-\lambda)^{1/2}),$
moreover
$$\partial_{\lambda}\phi^{\lambda}=O((e_{0}-\lambda)^{-1/2})\phi+o((e_{0}-\lambda)^{1/2}).$$
\end{lemma}

\section{The Linearized Operator}\label{sec:OperL}
After linearizing the solution around solitary wave solution \eqref{eq:soli2}, namely considering the solution of \eqref{eq:NLS} $\psi(x,t)=e^{i\lambda t} [\phi^{\lambda}(x)+R(x,t)]$, then the linear part of the equation for $R(x,t)$ is
\begin{align}
\partial_t \vec{R}=L(\lambda)\vec{R}
\end{align} with $\vec{R}:=(ReR,\ ImR)^{T},$
and the linearized operator
$L(\lambda)$ is defined as
\begin{align}\label{eq:defLLambda}
L(\lambda):=\left[
\begin{array}{cc}
0& L_{-}(\lambda)\\
-L_{+}(\lambda) &0
\end{array}
\right]
\end{align} with $L_{\pm}(\lambda)$ being linear Schr\"odinger operators defined as
\begin{align*}
L_{-}(\lambda):=-\Delta+V+\lambda+(\phi^{\lambda})^2,\ \text{and}\ L_{+}(\lambda):=-\Delta+V+\lambda+3(\phi^{\lambda})^2.
\end{align*}

By general result (Weyl's Theorem) on stability of the essential spectrum for localized perturbations of $J(-\Delta)$ 
\cite{RSIV}, 
$$\sigma_{ess}(L(\lambda))=(-i\infty,-i\lambda]\cap
[i\lambda,i\infty)$$ if the potential $V$ in Equation
\eqref{eq:NLS} decays at $\infty$ sufficiently rapidly.

Next we study the eigenvalues and eigenvectors of $L(\lambda).$ The proof can be found in \cite{GaWe}, hence we omit it.
\begin{lemma}\label{LM:NearLinear}
Let $L(\lambda)$, or more explicitly,  $L(\lambda(\delta),\delta)$
denote the linearized operator about the the bifurcating state
$\phi^\lambda, \lambda=\lambda(\delta)$. Note that $\lambda(0)=
e_0$. 

It has an eigenvector $\left(
\begin{array}{cc}
0\\
\phi^{\lambda}
\end{array}
\right)$ and an associated eigenvector $\left(
\begin{array}{cc}
\partial_{\lambda}\phi^{\lambda}\\
0
\end{array}
\right)$ with eigenvalue $0$:
\begin{align}
L(\lambda)\left(
\begin{array}{cc}
0\\
\phi^{\lambda}
\end{array}
\right)=0,\ \ \ \ L(\lambda)\left(
\begin{array}{cc}
\partial_{\lambda}\phi^{\lambda}\\
0
\end{array}
\right)=\left(
\begin{array}{cc}
0\\
\phi^{\lambda}
\end{array}
\right).
\end{align}

Corresponding to the (possibly degenerate) eigenvalue, $-e_1,\ -e_2,\ \cdots, \ -e_{N}$, of
$-\Delta+V$, the matrix operator $$L(\lambda=e_0,\delta=0)$$ has eigenvalues 
\begin{align}
\pm iE_{n}(e_0)=\pm i(e_0-e_n),\ n=1,2,\cdots, N.\label{eq:Enen}
\end{align}
For $\delta>0$ and small, these bifurcate to
(possibly degenerate) eigenvalues, of the operator $L(\lambda),$ $\pm iE_1(\lambda),\dots,$ $\pm
iE_N(\lambda)$ with eigenvectors
\begin{align}
\left(
\begin{array}{lll}
\xi_{1}\\
\pm i\eta_{1}
\end{array}
\right),\ \left(
\begin{array}{lll}
\xi_{2}\\
\pm i\eta_{2}
\end{array}
\right),\ \cdot\cdot\cdot, \left(
\begin{array}{lll}
\xi_{N}\\
\pm i\eta_{N}
\end{array}
\right)\label{eq:eigenf}
\end{align} with $\xi_n,\ \eta_n$ being real functions and
\begin{align}
\langle \phi^{\lambda},\ \xi_n\rangle=\langle \partial_{\lambda}\phi^{\lambda},\ \eta_n\rangle=0,\ \langle \xi_{n},\eta_{m}\rangle=\delta_{m,n}.\label{eq:orthogonality}
\end{align}
Moreover, for $\delta$ sufficiently small  $2E_n(\lambda)>\lambda,\
n=1,2,\cdot\cdot\cdot,N,$ (resonance at second order with
radiation).

\end{lemma}


Furthermore we need the following condition on the threshold resonances.
\begin{definition}
A function $h$ is called a threshold resonance function of
$-\Delta+V$ at $0$, the endpoint of the essential spectrum,
$|h(x)|\leq c\langle x\rangle^{-1_+}$ and $h$ is $C^{2}$ and
solves the equation
$$(-\Delta+V)h=0.$$

A function $h$ is called a threshold resonance function of
$L(\lambda)$ at $\mu=\pm i\lambda$, the endpoint of the essential spectrum, 
$|h(x)|\leq c\langle x\rangle^{-1_+}$ and $h$ is $C^{2}$ and
solves the equation
$$(L(\lambda)-\mu)h=0.$$
\end{definition}

In this paper we make the following assumption:
{\bf 
\begin{enumerate}
 \item[(SA)] $-\Delta+V$ has no threshold resonance at $0$.
\end{enumerate}
} This assumption is generic since it is known that the threshold resonance is unstable, see e.g. \cite{RSIII}. Based on this assumption it is well known that
\begin{lemma}
If $|e_0-\lambda|$ is sufficiently small and the assumption (SA) holds, then $L(\lambda)$ has no threshold resonances at $\mu=\pm i\lambda$, and $L(\lambda)$ has no other eigenvectors and eigenvalues besides the ones listed in Lemma \ref{LM:NearLinear}.

\end{lemma}


We denote the projection onto the essential spectrum of linear operator $L(\lambda)$ is
$P_c^{\lambda}=1-P_{d}^{\lambda}.$ In the following we give the explicit form of the projection $P_{d}$, whose proof for $N=1$ can be found in
~\cite{MR2187292}, the proof of the general cases are similar, hence
omitted.
\begin{proposition}\label{Riesz-project}
For the non self-adjoint operator $L(\lambda)$ the (Riesz)
projection onto the discrete spectrum subspace of $L(\lambda)$, $P_d=P_d(L(\lambda))=P_d^\lambda$, is given by
\begin{equation}\label{eq:PdProjection}
\begin{array}{lll}
P_{d}&=&\frac{2}{\partial_{\lambda}\| \phi^{\lambda}\|^{2}}\left(\ \left|
\begin{array}{lll}
0\\
\phi^{\lambda}
\end{array}
\right\rangle \left\langle
\begin{array}{lll}
0\\
\partial_{\lambda}\phi^{\lambda}
\end{array}
\right|\ +\ \left|
\begin{array}{lll}
\partial_{\lambda}\phi^{\lambda}\\
0
\end{array}
\right\rangle \left\langle
\begin{array}{lll}
\phi^{\lambda}\\
0
\end{array}
\right|\ \right)\\
& &\\
& &-i\displaystyle\sum_{n=1}^{N}\left(\ \left|
\begin{array}{lll}
\xi_{n}\\
i\eta_{n}
\end{array}
\right\rangle\left\langle
\begin{array}{lll}
-i\eta_{n}\\
\xi_{n}
\end{array}
\right| \ -\  \left|
\begin{array}{lll}
\xi_{n}\\
-i\eta_{n}
\end{array}
\right\rangle\left\langle
\begin{array}{lll}
i\eta_{n}\\
\xi_{n}
\end{array}
\right|\ \right).
\end{array}
\end{equation}
\end{proposition}
We define the projection onto the continuous spectral  subspace of $L(\lambda)$ by
\begin{align} 
P_c\ =\ P_c(L(\lambda))\ =\ P_c^\lambda\ \equiv\ I\ -\ P_d\label{Pcdef}.
\end{align}

\section{The Negativity of Fermi Golden Rule in Matrix Form}\label{sec:FGR}
The Fermi Golden rule plays an essential role in determining the
decay rate of the neutral modes. For simple neutral modes, as in
\cite{SW99, TsaiYau02, BuSu, SW04},
the form is simple because the number is only one, or if more than one then they
can be separated into independent ones by a near identity
transformation (see ~\cite{Tsai2003, Cucca08}). The problem of multiple neutral modes is more
involved due to the fact that multiple coupled parameters appear and
they can not be separated. 

Next we define the new Fermi Golden Rule condition.

Define a function $e:\mathbb{R}^3\times \mathbb{R}^3\rightarrow \mathbb{C}$ as
\begin{align}
e(x,k):=\big[1+ (-\Delta_x-|k|^2-i0)^{-1}V(x) \big]^{-1}e^{ik\cdot x}.
\end{align}
It is known that for any fixed $k\not=0$, this function is well defined by the type of potential $V$ we chose, see e.g. \cite{RSIII}. And it satisfies the equation
\begin{align*}
[-\Delta_x+V(x)-|k|^2]\ e(x,k)=0.
\end{align*}

We define complex functions $\Psi_{m,n}$ on the 2-dimensional unit sphere $\mathbb{S}^2$. For any $\sigma\in \mathbb{S}^2$,
\begin{align}
\Psi_{m,n}(\sigma)=\Psi_{n,m}(\sigma):= \int_{\mathbb{R}^3} e(x,\ |k|_{m,n} \sigma)\phi(x)\ \xi^{lin}_m(x)\xi^{lin}_n(x)\ dx\label{eq:Psimn}
\end{align} with $|k|_{m,n}\in \mathbb{R}^{+}$ defined as 
\begin{align*}
|k|_{m,n}:=(e_0-e_m-e_n)^{\frac{1}{2}},
\end{align*} where, recall that we assume that $2e_{l}<e_0,\ l=1,2,\cdots, N,$ and $\phi$, $\xi_{m}^{lin}$ are eigenvectors of $-\Delta+V$ with eigenvalues $-e_0$ and $-e_{m},\ m=1,\cdots,\ N.$

Now we state our Fermi-Golden-rule condition.
\begin{itemize}
\item[(FGR)] For any scalar vector $z=(z_1,z_2,\cdots,z_n)\in \mathbb{C}^{N}$ satisfying $|z|=1$, the functions defined on the unit sphere $\mathbb{S}^2$,
$ \displaystyle\sum_{m,n=1}^{N}\Psi_{m,n}(\sigma) z_{m}\bar{z}_{n} $ is not identically zero.
\end{itemize}

Its important ramification is that there exists some constant $C>0$, such that for any $z=(z_1,\ z_2,\cdots,\ z_n)\in \mathbb{C}^{N},$
\begin{align}
\Gamma(z,\bar{z}):=\| \sum_{m,\ n} \Psi_{m,n}(\sigma)z_m z_n\|_{\mathcal{L}^2(\mathbb{S}^2)}^2\geq C |z|^4. \label{Gammadef}
\end{align} Here we use that $\displaystyle\sum_{m,\ n} \Psi_{m,n}(\sigma)z_m z_n$ is smooth in $\sigma$, hence if it is not identically zero, we have the estimate above.

\begin{remark}

If the set of eigenvalues $\{e_{k}\ | k=1,\cdots,N\}$ can be grouped into well separated clusters, namely
\begin{align}
\{e_{k}|\ k=1,\cdots,N\}=\cup_{l} A_{l}
\end{align} with properties that, for some constant $c_0=\cO(1)$, $$|e_n-e_m|\geq c_0\ \text{if} \ e_n\in A_l, \ e_m\in A_k\ \text{with} \ l\not=k.$$ Then Fermi Golden rule assumption can be relaxed: for each fixed $l$, the each of the functions
$
 \displaystyle\sum_{
e_m,\ e_n\in A_l
}\Psi_{m,n}(\sigma)\ z_m z_n
$ is not identically zero, with $\displaystyle\sum_{e_m\in A_l}|z_m|^{2}\not=0.$

\end{remark}

\section{Main Theorem}\label{MainTHM}
In this section we state precisely the main theorem of this paper. The key fact is that despite of the possible presence of resonance, the main results remain the same as in \cite{GaWe,GaWe2011}.

For
technical reasons we impose the following conditions on the external potential $V$ of \eqref{eq:NLS}:
\begin{enumerate}
 \item[(VA)] $V$ decays exponentially fast at $\infty.$
\end{enumerate}

Recall the notations $\xi=(\xi_1,\cdots,\xi_N)$ and $\eta=(\eta_1,\dots,\eta_N)$ for components of the neutrally stable modes of frequencies $\pm iE_n(\lambda),\ n=1,\cdots,N,$ of the linearized operator $L(\lambda)$ defined in \eqref{eq:defLLambda}.

\begin{theorem}\label{THM:MainTheorem}
Suppose that Conditions (NL) in Section \ref{HaGWP}, (SA) in Section \ref{sec:OperL}, (FGR) in Section \ref{sec:FGR} and (VA) above are
satisfied. Let  $\nu>0$ be fixed and sufficiently large.
\\ Then there exist constants
$c,\epsilon_{0}>0$ such that, if
\begin{equation}\label{InitCond}
\inf_{\gamma\in
\mathbb{R}}
\left\|\psi_0-e^{i\gamma}
\left(\phi^{\lambda_{0}}+\ (Re\ z^{(0)})\cdot \xi+i\ (Im\ z^{(0)})\cdot\eta
\right)\right\|_{\mathcal{H}^{3,\nu}}\ \leq\ 
c\ | z^{(0)}| \le \epsilon_0,
\end{equation} 
  then there exist smooth functions 
  \begin{align}
  &\lambda(t):\mathbb{R}^{+}\rightarrow
\mathcal{I},\ \ \ \gamma(t): \mathbb{R}^{+}\rightarrow
\mathbb{R},\  z(t):\mathbb{R}^{+}\rightarrow \mathbb{C}^d,\nn\\
 &\ \  R(x,t):\mathbb{R}^{d}\times\mathbb{R}^{+}\rightarrow
\mathbb{C}
\nn\end{align}
 such that the solution of NLS evolves in the form:
\begin{align}\label{Decom}
\psi(x,t)&\ =\ e^{i\int_{0}^{t}\lambda(s)ds}e^{i\gamma(t)}\nn\\
&\ \ \ \ \ \ \ \times[\phi^{\lambda}+a_{1}(z,\bar{z})
\D_\lambda\phi^{\lambda}+ia_{2}(z,\bar{z})\phi^{\lambda} + (Re\ \tilde{z})\cdot\xi + i(Im \tilde{z})\cdot\eta + R ],\end{align}
where  $\lim_{t\rightarrow \infty}\lambda(t)=\lambda_{\infty},$
for some $\lambda_{\infty}\in \mathcal{I}$.\\ 
Here, $a_{1}(z,\bar{z}),\ a_{2}(z,\bar{z}): \mathbb{C}^N\times\mathbb{C}^N\rightarrow \mathbb{R}$ and $\tilde{z}-z: \mathbb{C}^N\times\mathbb{C}^N\rightarrow \mathbb{C}^N$
 are polynomials of $z$ and $\bar{z}$, beginning with terms of order $|z|^{2}$. Their explicit definitions will be given in \eqref{eq:pkmn}.\\
Moreover:
\begin{enumerate}
\item[(A)] $|z(t)|\leq c(1+t)^{  -\frac{1}{2}  }$ and, there exists a polynomial $F(z,\bar{z})=O(|z|^4)\in \mathbb{R}$ such that
$z$ satisfies the initial value problem 
\begin{equation}\label{eq:detailedDescription1}
{\partial_t}[ |z|^2+F(z,\bar{z})] =-C\Gamma( z ,\bar{ z })  +\ \cO((1+t)^{-\frac{12}{5}})
\end{equation} where $C>0$ is a constant, $\Gamma( z ,\bar{ z })=\cO(|z|^4)$ is a positive quantity defined in \eqref{Gammadef}. 
 \item[(B)] $\vec{R}(t)=( Re R(t), Im R(t))^T$ lies in the essential spectral part of $L(\lambda(t))$. Equivalently, $R(\cdot,t)$ satisfies the symplectic orthogonality conditions:
 \begin{align}\label{s-Rorthogonal}
&\omega\langle R,i\phi^{\lambda}\rangle\ =\ \omega\langle
R,\partial_{\lambda}\phi^{\lambda}\rangle\ =\ 0\nn\\
&\omega\langle
R,i\eta_{n}\rangle=\omega\langle R,\xi_{n}\rangle=0,\
n=1,2,\cdot\cdot\cdot,N,
\end{align} 
 where $\omega(X,Y)=Im\ \int X\bar{Y}$
 and satisfies the decay estimate:
 \begin{equation}
\|(1+x^{2})^{-\nu}\vec{R}(t)\|_{2}\leq c(1+|t|)^{-1}.
\label{Rdecay}
\end{equation}
\end{enumerate}
\end{theorem}
The main theorem will be reformulated into Theorem \ref{GOLD:maintheorem} below.

\section{The Effective
Equations for $\dot{z},\ \dot\lambda,\ \dot\gamma$ and $
R$}\label{SEC:effective}
In this section we derive equations for $\dot{z},\ \dot\lambda,\ \dot\gamma$ and $R$.

We decompose the solution as 
\begin{align}
\psi(x,t)&\ =\ e^{i\int_{0}^{t}\lambda(s)ds}e^{i\gamma(t)}\nn\\
&\ \ \ \ \ \times\left[\phi^{\lambda}+a_{1}
\phi_{\lambda}^{\lambda}+ia_{2}\phi^{\lambda}+\sum_{n=1}^{N}(\alpha_{n}+p_{n})\xi_{n}+
i\sum_{n=1}^{N}(\beta_{n}+q_{n})\eta_{n}+R\right]\nn\\
&\ =\ e^{i\int_{0}^{t}\lambda(s)ds}e^{i\gamma(t)}\left[\phi^{\lambda}+a_{1}
\phi_{\lambda}^{\lambda}+ia_{2}\phi^{\lambda}+(\alpha+p)\cdot\xi+
i(\beta+q)\cdot\eta+R\right]
\label{Decom1}
\end{align} 
Here and going forward, we'll use the notations:
\begin{align}
\alpha=(\alpha_1,\dots,\alpha_N)^T,\ \ \beta=(\beta_1,\dots,\beta_N)^T,\nn\\
\xi=(\xi_1,\dots,\xi_N)^T,\ \ \eta=(\eta_1,\dots,\eta_N)^T.\nn
\end{align}
Introducing 
\begin{equation} z=\alpha+i\beta,\nn\end{equation}
we have 
\begin{equation}
\alpha=\frac{1}{2}(z+\overline{z}),\ \ \beta=\frac{1}{2i}(z-\overline{z}),
\nn\end{equation}
and we seek $a_j=a_j(z,\overline{z})=\cO(|z|)^2,\ j=1,2$ and $p_j=p_j(z,\overline{z})=\cO(|z|^2)$, polynomials in $z$ and $\overline{z}$, which are of degree larger than or equal to two, and are real.
Substitution of the Ansatz \eqref{Decom1} into NLS,
equation \eqref{eq:NLS}, we have the following system of equations
for  $\vec{R}$, defined as
\begin{align}
\vec{R}:= \left(
\begin{array}{lll}
R_{1}\\
R_{2}
\end{array}
\right),\ R_{1}:= Re R,\ \ R_{2}:= Im R:
\nn
\end{align}
as
\begin{align}
\partial_{t}\vec{R}=&L(\lambda)\vec{R}+\dot\gamma
J\vec{R}-J\vec{N}(\Vec{R},z)
-\left(
\begin{array}{lll}
\partial_{\lambda}\phi^{\lambda}[\dot\lambda+\partial_{t}a_{1}]\\
\phi^{\lambda} [ \dot{\gamma}+\partial_{t} a_{2}-a_{1} ]
\end{array}
\right)\nn\\
& +\ \left(
\begin{array}{lll}
\xi\cdot[E(\lambda)(\beta+q)-\partial_{t}(\alpha+p)]\\
-\eta\cdot[E(\lambda)(\alpha+p)+\partial_{t}(\beta+q)]
\end{array}
\right)\label{Eq:R} \\
&+\dot\gamma\ \left(
\begin{array}{lll}
(\beta+q)\cdot\eta\\
- (\alpha+p)\cdot\xi
\end{array}\right)
-\dot\lambda \left(
\begin{array}{lll}
a_{1}\partial_{\lambda}^{2}\phi^{\lambda}+
(\alpha+p)\cdot\partial_{\lambda}\xi\\
 a_{2}\partial_{\lambda}\phi^{\lambda}+
(\beta+q)\cdot\partial_{\lambda}\eta
\end{array}
\right),\nonumber
\end{align}

Here, 
\begin{align}
J\vec{N}(\Vec{R},z):=\left(
\begin{array}{lll}
ImN(\vec{R},z)\\
-ReN(\vec{R},z)
\end{array}
\right) \label{JvecN}
\end{align}
with
\begin{align}
ImN(\Vec{R},z)&:=
|\phi^{\lambda}+I_{1}+iI_{2}|^{2} I_{2}-(\phi^{\lambda})^{2} I_{2}\nn\\
Re N(\Vec{R},z)&:=
[|\phi^{\lambda}+I_{1}+iI_{2}|^{2}-(\phi^{\lambda})^{2}](\phi^{\lambda}+I_{1})
-2(\phi^{\lambda})^{2}I_{1}\nn\\
I_{1}\ &=\ A_{1}+A_{2}+R_{1},\ \ \ \ 
I_{2}\ =\ B_{1}\ +\ B_{2}\ +\ R_{2}\nn\\
A_{1}\ &=\ \alpha\cdot\xi, \ \ \ \ \ A_{2}\ =\ a_{1}\partial_{\lambda}\phi^{\lambda}+
p\cdot\xi,\nn\\
 B_{1}\ &=\ \beta\cdot \eta,\ \ \ \ \ 
B_{2}\ =\ a_{2}\phi^{\lambda}\ +\ q\cdot\eta.\label{eq:A12B12}
\end{align}

From the system of equations \eqref{Eq:R} and the orthogonality conditions
\eqref{eq:orthogonality} and \eqref{s-Rorthogonal} we obtain equations for $
\dot\lambda,\ \dot\gamma$ and $z_{n}=\alpha_n+i\beta_n,\ n=1,\dots,N: $
\begin{align}
\partial_{t}(\alpha_{n}+p_{n})-E_n(\lambda)(\beta_{n}+q_{n})+\langle ImN(\Vec{R},z),\eta_{n}\rangle=&F_{1n};\label{eq:z1} \\
\partial_{t}(\beta_{n}+q_{n})+E_n(\lambda)(\alpha_{n}+p_{n})-\langle ReN(\Vec{R},z),\xi_{n}\rangle=&F_{2n};\label{eq:z2}\\
\dot\gamma+\partial_{t}a_{2}-a_{1}-\frac{1}{\langle
\phi^{\lambda},\phi^{\lambda}_{\lambda}\rangle}\langle
ReN(\Vec{R},z),\phi^{\lambda}_{\lambda}\rangle=& F_{3};\label{eq:gamma}\\
\dot\lambda+\partial_{t}a_{1}+\frac{1}{\langle
\phi^{\lambda},\phi^{\lambda}_{\lambda}\rangle}\langle
ImN(\Vec{R},z),\phi^{\lambda}\rangle=&F_{4}.\label{eq:lambda}
\end{align}
Finally, the scalar functions $F_{j,n},\ j=1,2, F_3,
 F_4,$ are defined as
\begin{align}
F_{1n}:= &\dot\gamma\langle
(\beta+q)\cdot\eta,\eta_{n}\rangle-\dot{\lambda}a_{1}\langle
\partial_{\lambda}^{2}\phi^{\lambda},\eta_{n}\rangle
-\dot\lambda\langle
(\alpha+p)\cdot\partial_{\lambda}\xi,\eta_{n}\rangle\nn\\
&-\dot\gamma\langle
R_{2},\eta_{n}\rangle
+\dot\lambda\langle
R_{1},\partial_{\lambda}\eta_{n}\rangle,\nn\\
F_{2n} := &-\dot\gamma\langle
(\alpha+p)\cdot\xi,\xi_{n}\rangle
-\dot{\lambda}a_{2}\langle\phi^{\lambda}_{\lambda},
\xi_{n}\rangle
-\dot\lambda\langle(\beta+q)\cdot\partial_{\lambda}\eta,\xi_{n}\rangle\nn\\
&+\dot\gamma\langle
R_{1},\xi_{n}\rangle+\dot\lambda\langle
R_{2},\partial_{\lambda}\xi_{n}\rangle,\nn\\
F_{3} := &\frac{1}{\langle
\phi^{\lambda},\phi^{\lambda}_{\lambda}\rangle}\times\nn\\
&
\left[\ \dot\lambda \langle
R_{2},\phi_{\lambda\lambda}^{\lambda}\rangle -\dot\gamma\langle
R_{1},\phi_{\lambda}^{\lambda}\rangle
-\langle\dot\gamma(\alpha+p)\cdot\xi+
\dot{\lambda}a_{2}\phi^{\lambda}_{\lambda}
+\dot\lambda (\beta+q)\cdot\partial_{\lambda}\eta,\phi^{\lambda}_{\lambda}\rangle\ \right],\nn\\
F_{4}:=&\frac{1}{\langle
\phi^{\lambda},\phi^{\lambda}_{\lambda}\rangle}\times\nn\\
&
\left[\ \dot\lambda\langle
R_{1},\phi_{\lambda}^{\lambda}\rangle +\dot\gamma\langle
R_{2},\phi^{\lambda}\rangle
+\langle
\dot\gamma (\beta+q)\cdot \eta-\dot{\lambda}a_{1}
\partial_{\lambda}^{2}\phi^{\lambda}-\dot\lambda
(\alpha+p)\cdot\partial_{\lambda}\xi,\phi^{\lambda}\ \rangle\right].
\end{align}

To facilitate later discussions we cast \eqref{eq:z1} and \eqref{eq:z2} into a convenient form. Since $\alpha_n$ and $\beta_n$ are real parameters, it is equivalent to study the complex parameters $z_n:=\alpha_n+i\beta_n$. Compute \eqref{eq:z1}$+i$\eqref{eq:z2} to find
\begin{align}
\partial_t z_n+iE_{n}(\lambda)z_n=&-\partial_{t} (p_n+iq_n)-iE_{n}(\lambda)(p_n+iq_n)-\langle ImN(\Vec{R},z),\eta_{n}\rangle\nonumber\\
&+i\langle ReN(\Vec{R},z),\xi_{n}\rangle
+F_{1n}+iF_{2n}.\label{eq:z}
\end{align}
Note that \eqref{eq:z1} and \eqref{eq:z2} can be recovered from the equation above by considering the real and imaginary parts of \eqref{eq:z}.

\begin{remark}
\begin{itemize}
\item[(a)] Recall the estimate of $Remainder$ in \eqref{remainder}. By \eqref{eq:z1}-\eqref{eq:lambda} we have
\begin{equation}\label{eq:EstRough}
\dot\lambda,\ \dot\gamma,\
\partial_{t}z_{n}+iE_n(\lambda)z_{n}=O(|z|^{2})+Remainder.
\end{equation}
\item[(b)] The functions $a_j(z,\overline{z}), j=1,2,\ p_n(z,\overline{z}),\ 
 q_n(z,\overline{z}), n=1,\dots,N$ will be chosen to eliminate ``non-resonant'' terms: $z^a\overline{z}^b, \ \ 2\le a+b\le3$.
 \end{itemize}
\end{remark}
Finally, we derive an equation for 
\begin{equation}
\vec{R}=P_c^{\lambda(t)}\vec{R}=P_c\vec{R},
\nn\end{equation}
the continuous spectral part of the solution, relative to the operator, $L(\lambda(t))$. Applying $P_c=P_c^{\lambda(t)}$ to \eqref{Eq:R} to use that (see \eqref{eq:PdProjection})
\begin{align*}
P_c\left(
\begin{array}{lll}
\xi_n\\
\pm i \eta_n
\end{array}
\right)=P_c\left(
\begin{array}{lll}
0\\
\phi^{\lambda}
\end{array}
\right)=P_c\left(
\begin{array}{lll}
\partial_{\lambda}\phi^{\lambda}\\
0
\end{array}
\right)=0
\end{align*} to remove many terms on the right hand side,
and using the commutator identity:
\begin{align}
P_{c}\partial_t\vec{R}=\partial_t\vec{R}-\dot\lambda
\D_\lambda P_{c}\vec{R}
\label{Pccomdt}
\end{align}
we obtain
\begin{align}\label{RAfProj}
\partial_t\vec{R}\ =\ 
 L( \lambda(t) )\vec{R}\ -\ P_{c}^{\lambda(t)}J\vec{N}(\Vec{R},z)\ +\ L_{(\dot\lambda,\dot\gamma)}\vec{R}\ +\ \mathcal{G}.
\end{align}
Here the operator $L_{(\dot\lambda,\dot\gamma)}$ and the 
 vector function $\mathcal{G}$ are defined 
as 
\begin{align}
& L_{(\dot\lambda,\dot\gamma)}\ =\ \dot\lambda
\D_\lambda P_{c}^{\lambda(t)}+\dot\gamma P_{c}^{\lambda(t)}J,\label{Lgdld}\\
& \mathcal{G}\ =\ P_{c}^{\lambda(t)}\left(
\begin{array}{lll}
\dot\gamma (\beta+q)\cdot\eta-\dot{\lambda}a_{1}
\partial_{\lambda}^{2}\phi^{\lambda}
-\dot\lambda (\alpha+p)\cdot\partial_{\lambda}\xi\\
-\dot\gamma (\alpha+p)\cdot\xi-\dot{\lambda}a_{2}\phi^{\lambda}_{\lambda}-\dot\lambda
(\beta+q)\cdot\partial_{\lambda}\eta
\end{array}
\right).\label{Gdef}
\end{align}
\section{Reformulation of The Main Theorem}\label{sec:refor}
 The proof of Theorem ~\ref{THM:MainTheorem} we use the following result, which gives a more detailed characterization of the terms in the decomposition.
\begin{theorem}\label{GOLD:maintheorem}
The function $R$ in \eqref{Decom} of Theorem ~\ref{THM:MainTheorem}
can be decomposed as
\begin{equation}\label{eq:expanR}
\vec{R}=\displaystyle\sum_{ |m|+|n|=2}R_{m, n}(\lambda)z^{m}\bar{z}^n+\tilde{R}
\end{equation} where $R_{m,n}$ are functions of the form
$$R_{m,n}=\big[L(\lambda)+iE(\lambda)\cdot (m-n)-0\big]^{-1}\phi_{m,n}$$
$\phi_{m,n}$ are smooth, spatially exponentially decaying functions. The function
$\tilde{R}$ satisfies the equation
\begin{equation}\label{eq:tildeR}
\begin{array}{lll}
\partial_t\tilde{R}&=&L(\lambda)\tilde{R}+M_{2}(z,\bar{z})\tilde{R}+P_{c}N_{2}(\vec{R},z)+P_{c}S_{2}(z,\bar{z}),
\end{array}
\end{equation} where
\begin{enumerate}
\item[(1)] $S_{2}(z,\bar{z})=\cO(|z|^{3})$ is a polynomial in $z$ and
 $\bar{z}$ with $\lambda$-dependent coefficients, and each coefficient can be written as the sum of
 functions of the form
 \begin{align}\label{eq:unusual}
[L(\lambda)\pm i(E_m(\lambda)+E_n(\lambda))-0]^{-k}P_{c}\phi_{\pm k}(\lambda),
\end{align}
where $k=0,1,2$ and the functions $\phi_{\pm k}(\lambda)$ are smooth
and decay exponentially fast at spatial $\infty$;
 \item[(2)] $M_{2}(z,\bar{z})$ is an operator defined by
 \begin{equation}\label{eq:M2}
 M_{2}(z,\bar{z}):=\dot{\gamma}P_{c}J+\dot{\lambda}P_{c\lambda}+X,
 \end{equation}
 where $X$ is a $2\times 2$ matrix, satisfying the bound
   $$|X|\leq c|z|e^{-\epsilon_{0}|x|}.$$
 \item[(3)] $N_{2}(\vec{R},z)$ can be separated into
 localized term and nonlocal term
 \begin{equation}\label{eq:N2Decomposition}
 N_{2}=Loc+NonLoc
 \end{equation} where $Loc$ consists of terms spatially
 localized (exponentially) function of $x\in \mathbb{R}^{3}$ as a factor and satisfies the estimate
 \begin{equation}
 \|\langle x\rangle^{\nu}(-\Delta+1)Loc\|_{2}+\|Loc\|_{1}+
 \|Loc\|_{\frac{4}{3}}\leq c(|z|^{3}(t)+|z|(t)\|\langle x\rangle^{-\nu}\vec{R}\|_{2}),
 \label{eq:Loc-est}\end{equation}
  and $NonLoc$ is given by
\begin{equation}\label{eq:NonlocDef}
NonLoc:=(R_{1}^{2}+R_{2}^{2})J\vec{R}.
\end{equation}
Here $\nu$ is the same as in Theorem
~\ref{THM:MainTheorem}.
\end{enumerate}

In the rest of the paper we denote by $Remainder(t)$ any quantity which satisfies the estimate:
\begin{align}\label{remainder}
|Remainder(t)|\lesssim |z(t)|^{4}+\|\langle
x\rangle^{-\nu}\vec{R}(t)\|_{2}^{2}+\|\vec{R}(t)\|_{\infty}^{2}
+|z(t)|\  \|\langle
x\rangle^{-\nu}\tilde{R}(t)\|_{2}.
\end{align} 

The functions $\lambda$, $\gamma,$ $z$ have the following properties
\begin{enumerate}
\item[(A)]
\begin{equation}\label{ExpanLambda}
\dot\lambda=Remainder(t);
\end{equation}
\item[(B)]
\begin{equation}\label{ExpanGamma}
\dot\gamma=\Upsilon+Remainder(t)
\end{equation} with
\begin{equation}\label{eq:Gamma11}
\Upsilon:=\frac{\langle
(\phi^{\lambda})^{2}[\frac{3}{2}|z\cdot\xi|^{2}+\frac{1}{2}|z\cdot\eta|^{2}],\partial_{\lambda}\phi^{\lambda}\rangle}{\langle\phi^{\lambda},\partial_{\lambda}\phi^{\lambda}\rangle};
\end{equation} 
\item[(C)]
there exists a polynomial $F(z,\bar{z})=\cO(|z|^4)\in \mathbb{R}$ such that the vector $z$ satisfies the equation
\begin{equation}\label{eq:detailedDescription}
{\partial_t}[|z|^2+F(z,\bar{z})]=-\Gamma(z,\bar{z})+|z|Remainder(t)
\end{equation} where $\Gamma(z,\bar{z})$ is a positive quantity defined in \eqref{Gammadef}.\end{enumerate}

\end{theorem}

The definition of $R_{m,n}$ in \eqref{eq:expanR} will be in Section \ref{SEC:effective}, the proof of \eqref{eq:tildeR} will be in Section \ref{sec:tildeR}, \eqref{ExpanLambda} and \eqref{ExpanGamma} will be reformulated into Proposition \ref{Prop:ExplicitPolyno}, \eqref{eq:detailedDescription} will be proved in Section \ref{sec:NFT}.

\section{Proof of \eqref{eq:tildeR}}\label{sec:tildeR}
Observe that in the equation for $\Vec{R}$ in \eqref{RAfProj}, the term on the right hand side, specifically $J\Vec{N}(\Vec{R},z)$, contains terms quadratic in $z$ and $\bar{z}$.
Hence for fixed $z(t)\in\C^N$, the equation for $\vec{R}(t)$ is forced by terms
of order $\cO(|z(t)|^2)$.  

In what follows, we extract the quadratic in $z,\overline{z}$ part of $\vec{R}(t)$.
Observe that the quadratic terms generated by the nonlinearity are of the form:
\begin{align}
\sum_{|m|+|n|=2}\ JN_{m,n}z^{m}\bar{z}^n\ =&\left(
\begin{array}{ccc}
2\phi^{\lambda}A_{1}B_{1}\\
-3\phi^{\lambda}A_{1}^{2}-\phi^{\lambda}B_{1}^{2}
\end{array}\right).\label{eq:SecondOrderTerm}
\end{align} 
where $A_{1}=\alpha\cdot\xi,\ B_{1}=\beta\cdot\eta$, and recall the definition of $JN$ from \eqref{JvecN}.
Substitute this into the equation for $\Vec{R}$ in 
\eqref{RAfProj} and decompose $\Vec{R}$ in the next results:
\begin{theorem}
Define 
\begin{equation}\label{eq:Rform}
R_{m,n}:=[L(\lambda)+iE(\lambda)\cdot (m-n)-0]^{-1}P_{c}JN_{m,n},
\end{equation}
and decompose $\vec{R}(t)$ as
\begin{equation}
\Vec{R}\ =\ \sum_{|m|+|n|=2}\ R_{m,n}z^{m}\bar{z}^{n}\ +\ \tilde{R}
\label{RtR}
\end{equation}
Then the vector-function $\tilde{R}(x,t)$
 satisfies \eqref{eq:tildeR}.
\end{theorem}

The proof is the same to that in \cite{GaWe}, and skipping it will not affect understanding the main part of this paper. Hence we choose to omit this part.

To facilitate later discussions, we further decompose $$J\vec{N}_{>2}=J\vec{N}(\vec{R},z)-\sum_{|m|+|n|=2}\ JN_{m,n}z^{m}\bar{z}^n.$$ We extract the third order terms of $J\vec{N}_{>2}:$
\begin{equation}\label{eq:JNm+n=3}
\sum_{|m|+|n|=3}JN_{m,n}z^{m}\bar{z}^{n}=\sum_{|m|+|n|=2}X R_{m,n}z^m\bar{z}^{n}+X\left(
\begin{array}{lll}
A_{2}\\
B_{2}
\end{array}
\right)+\left(
\begin{array}{lll}
(A_{1}^{2}+B_{1}^{2})B_{1}\\
-(A_{1}^{2}+B_{1}^{2})A_{1}
\end{array}
\right)
\end{equation} where
$X$ is a $2\times 2$ matrix defined as
\begin{equation}\label{eq:XDef}
X:=\left(
\begin{array}{lll}
2\phi^{\lambda}B_{1}& 2\phi^{\lambda}A_{1}\\
-6\phi^{\lambda}A_{1}&
-2\phi^{\lambda}B_{1}
\end{array}
\right),
\end{equation} and $A_1,\ B_1$ and $A_2,\ B_2$ are defined in \eqref{eq:A12B12}.

\section{Normal Form Transformon, Proofs of \eqref{ExpanLambda} and \eqref{ExpanGamma}}\label{NormalForms}
In this section we present the proofs of equations \eqref{ExpanLambda} and \eqref{ExpanGamma}, governing $\dot\lambda$ and $\dot\gamma$, crucial to controlling the large time behavior. The main result is Proposition \ref{Prop:ExplicitPolyno}.

This part is different from \cite{GaWe, GaWe2011}, in that we have to define a new normal form transformation, some of whose parameters are defined as solutions to systems of linear equations and their existence has to be justified.
Moreover some small denominators will appear and we have to prove the numerators are also small.

Now we present the idea. Central to our claim about the large time dynamics of NLS, is that the solution settles into
an asymptotic solitary wave, $\phi^{\lambda_\infty}$, where $\lambda(t)\to\lambda_\infty$. We achieve this by showing $|\dot\lambda(t)|\lesssim \epsilon_0 (1+t)^{-1-\delta}$ for some $\delta>0$ and small $\epsilon_0>0.$
Since we expect the neutral mode amplitudes, $z(t)$, to decay with a rate $t^{-\frac{1}{2}}$, we require that there be no $\cO(|z(t)|^2)$ on the right hand side of the equation \eqref{eq:lambda}:
  $$\dot\lambda(t) = -\D_t a_1(z,\overline{z})-\frac{1}{\langle
\phi^{\lambda},\phi^{\lambda}_{\lambda}\rangle}\langle
ImN(\Vec{R},z),\phi^{\lambda}\rangle+ \dots.$$ The strategy is
  to choose the quadratic part of the polynomial $a_1(z,\overline{z})$
  so as to eliminate all quadratic terms. There are two types of such terms, (1) the terms $z_k z_l$ and $\bar{z}_k \bar{z}_l$,  and they are oscillatory with frequencies $\sim -E_k-E_l$ or $\sim E_k+E_l$, which stay away from zero. And the margins are large enough so that we can easily remove them by a normal form transformation, utilizing that $z_k z_l\approx \partial_{t} \frac{1}{-i (E_k+E_l)} z_k z_l.$ (2)
The terms $z_k\overline{z_m}$ have frequencies $\sim -E_k+E_m$, which might be of small, or zero, frequencies. The key observation is that is that if the frequency of $z_{k}\bar{z}_l$ is small, then the coefficient is proportionally small! This allows us to define normal form transformation. The calculation is carried out below; see
Lemma~\ref{coordinate-lemma}, especially \eqref{defA11}.
  
Similarly we choose $p_n,\ q_n,\ n=1,2,\cdots,N,$ and $a_2$ to remove most of the lower order terms in the equations for $\dot\gamma,\ \dot\alpha_n$ and $\dot\beta_n$. It turns out some terms can not be removed, for example the term $|z_n|^{2}$ in the equation for $\dot\gamma$, and $|z_n|^2 z_n$ terms in the equation for $\partial_{t}z$. On the other hand, there terms either play a favorable role in our analysis, or will not affect it, namely it does not matter if $\gamma$ is not convergent, since $e^{i\gamma(t)}$ is only a phase factor. 

In defining $p_n,\ q_n$ we have to solve system of linear equations, see e.g. \eqref{eq:pq20and02}. The existence and uniqueness of solutions have to be addressed.

In what follows we use the notations $N^{Im}_{m,n},\ N^{Re}_{m,n}$
to stand for functions satisfying 
$$\left(
\begin{array}{lll}
N^{Im}_{m,n}\\   
-N^{Re}_{m,n}
\end{array}
\right)=JN_{m,n},$$ where, recall the definition of $JN_{m,n},\ |m|+|n|=2,3,$ from \eqref{eq:SecondOrderTerm} and \eqref{eq:JNm+n=3}.

In what follows define the polynomials $a_{1}$, $a_{2}$, $p_{k}$ and $q_{k},\
k=1,2,\cdot\cdot\cdot,N$ in \eqref{Decom} by defining their coefficients:
\begin{align}
a_{1}(z,\bar{z}):=&\displaystyle\sum_{|m|+|n|=2,3}A^{(1)}_{m,n}(\lambda)z^{m}\bar{z}^{n},
\nonumber\\
a_{2}(z,\bar{z}):=&\displaystyle\sum_{|m|+|n|=2,3,\ |m|\not=|n|}A^{(2)}_{m,n}(\lambda)z^{m}\bar{z}^{n},\nonumber\\
p_{k}(z,\bar{z}):=&\displaystyle\sum_{|m|+|n|=
2,3}P^{(k)}_{m,n}(\lambda)z^{m}\bar{z}^{n}, \label{eq:pkmn}\\
q_{k}(z,\bar{z}):=&\displaystyle\sum_{|m|+|n|=
2,3}Q^{(k)}_{m,n}(\lambda)z^{m}\bar{z}^{n},\nonumber
\end{align}
with $m,n\in (\mathbb{Z}^{+}\cup \{0\})^{N}$.
In what follows we use the notation $$m\cdot E(\lambda)=\sum_{k}m_kE_k(\lambda).$$

We start with defining $A_{m,n}^{(1)}$.
For $|m|=2,3$
\begin{align}
A_{m,0}^{(1)}:=&\frac{1}{im\cdot E(\lambda)}\frac{1}{\langle
\phi^{\lambda},\partial_{\lambda}\phi^{\lambda}\rangle}\langle
N^{Im}_{m,0},\phi^{\lambda}\rangle.
\end{align} For $|m|=2,\ |n|=1,$
\begin{align}
A_{m,n}^{(1)}:=&\frac{1}{\langle
\phi^{\lambda},\partial_{\lambda}\phi^{\lambda}\rangle}\frac{1}{i (m-n)\cdot E(\lambda)}\times \nn\\
&\big[\langle
N^{Im}_{m,n},\phi^{\lambda}\rangle-\frac{i}{2}\sum_{{\small{
\begin{array}{lll}
|k|=|r|=|n|=1,\\ 
k+r=m
\end{array}
}}}\Upsilon_{k,n}\langle
r\cdot\eta,\phi^{\lambda}\rangle\big], \label{eq:A1}
\end{align}
here $\Upsilon_{m,n}$ is from the expansion of $\Upsilon=\Upsilon(z,\bar{z})$, defined in \eqref{eq:Gamma11}:
\begin{align*}
\Upsilon=\sum_{|m|=|n|=1} \Upsilon_{m,n}z^{m}\bar{z}^{n}
\end{align*} and the vector $r $ is in $(\mathbb{Z}^{+}\cup \{0\})^{N}.$

For $|m|= |n|=1$, we define
\begin{align}
A_{m,n}^{(1)}:=&\frac{1}{\langle
\phi^{\lambda},\partial_{\lambda}\phi^{\lambda}\rangle}\frac{1}{i (m-n)\cdot E(\lambda)}\langle
N^{Im}_{m,n},\phi^{\lambda}\rangle\label{defA11}\\
=&\frac{1}{4} \frac{1}{\langle
\phi^{\lambda},\partial_{\lambda}\phi^{\lambda}\rangle}[\langle n\cdot \eta, \ m\cdot\eta\rangle+\langle n\cdot\xi,\ m\cdot \xi\rangle].
\end{align}
Here in the first line it is possible that the denominator $(m-n)\cdot E(\lambda)$ equals to zero or arbitrarily small, for example $m=n$, which might make the term ill-defined. In the second line we indicate that, by using Lemma \ref{coordinate-lemma} below, if the denominator is small, then the numerator is proportionally small. Hence $A_{m,n}^{(1)}$ are always well defined, a similar calculation was in \cite{GaWe2011}.

After defining various terms above, the other terms in $A^{(1)}_{m,n}$ are determined by the relations
$$A_{k,l}^{(1)}:=\overline{A_{l,k}^{(1)}}\ \text{for}\ |k|+|l|=2,3.$$

Now we define various terms in $A^{(2)}_{m,n},\ |m|+|n|=2,3, \ |m|\not=|n|$. 

For $|m|=2,3,\ |n|=0$, we define $A_{m,0}^{(2)}$ as
\begin{align}
-i m\cdot E(\lambda)A_{m,0}^{(2)}+A_{m,0}^{(1)}:=&\frac{1}{\langle
\phi^{\lambda},\partial_{\lambda}\phi^{\lambda}\rangle}\langle
N^{Re}_{m,0},\partial_{\lambda}\phi^{\lambda}\rangle.
\end{align}
For $|m|=2,\ |n|=1$,
\begin{align}
&-i(m-n)\cdot E(\lambda)A_{m,n}^{(2)}+A_{m,n}^{(1)}\nn\\
:=&\frac{1}{\langle
\phi^{\lambda},\partial_{\lambda}\phi^{\lambda}\rangle}[\langle
N^{Re}_{m,n},\partial_{\lambda}\phi^{\lambda}\rangle
-\frac{1}{2}\sum_{|k|=|n|=|r|=1, k+r=m}\Upsilon_{k,n}\langle
r\cdot\xi,\partial_{\lambda}\phi^{\lambda}\rangle],
\end{align} here the definition of $\Upsilon_{m,n}$ and the convention on $r$ are as in \eqref{eq:A1}. Especially in solving for $A_{m,n}^{(2)}$ we need $\frac{1}{(m-n)\cdot E(\lambda)}$ to be uniformly bounded, indeed this is true by the facts $E_{k}(\lambda)\approx e_0-e_{k}$ in \eqref{eq:Enen} and $2e_k<e_0$, see Condition (NL) in Section \ref{HaGWP}.

The other terms in $A_{k,l}^{(2)}$ are determined by the relations
$$A_{k,l}^{(2)}:=\overline{A_{l,k}^{(2)}}\ \text{for}\ |k|+|l|=2,3.$$

Next we define coefficients $P^{(k)}_{m,n}$ and $Q^{(k)}_{m,n}$ for the polynomials $p_k$ and $q_k,\ k=1,2,\cdots, N,$ see \eqref{eq:pkmn}.

For $|m|=2,3,\ |n|=0,$ we define $P_{m,0}^{(k)}$ and $Q_{m,0}^{(k)}$ to be solutions to the linear equations
\begin{align}
-im\cdot E(\lambda)P_{m,0}^{(k)}-E_k(\lambda)Q_{m,0}^{(k)}:=&-\langle
N_{m,0}^{Im},\eta_{k}\rangle,\nonumber\\
-i m\cdot E(\lambda)Q_{m,0}^{(k)}+E_k(\lambda)P_{m,0}^{(k)}:=&\langle
N_{m,0}^{Re},\xi_{k}\rangle.\label{eq:pq20and02}
\end{align} Here the solutions exist and are unique because the corresponding $2\times 2$ matrices 
\begin{align}
D:=\left[
\begin{array}{cc}
-im\cdot E(\lambda) & -E_k(\lambda)\\
E_k(\lambda) & -im \cdot E(\lambda)
\end{array}
\right]\label{Dmatrix}
\end{align}
are uniformly invertible, or equivalently they determinants stay away from zero by a uniform margin. To see this, compute directly to obtain
\begin{align}
DetD=-[m\cdot E(\lambda)]^2+E_{k}^2(\lambda)=-[m\cdot E(\lambda)+E_k(\lambda)]\ [m\cdot E(\lambda)-E_k(\lambda)].
\end{align} Next we relate the quantities on the right hand side to $e_0$ and $e_l,\ l=1,2,\cdots,N,$ by \eqref{eq:Enen}, namely  $$\text{for any}\ l=1,2,\cdots, N,\ E_{l}(\lambda)\approx e_0-e_l.$$ Compute directly and use $2e_{l}< e_0$ to see 
\begin{align}
m\cdot E(\lambda)-E_k(\lambda)&\approx \sum_{l=1}^{N} m_l(e_0-e_l)-(e_0-e_k)\nonumber\\
&=(|m|-1)e_0+e_k-\sum_{l=1}^{N} m_l e_l \label{eq:awayZero}
\end{align} is positive and stays way from zero for any $m\in \{\mathbb{Z}^{+}\cup \{0\}\}^{N}$ and $|m|=2,3$.
This together with that $m\cdot E(\lambda)+E_k(\lambda)$ is more positive implies the desired result: $Det D$ stays away from zero.

For $|m|=1$ and $|n|=2,$ we define $P_{m,n}^{(k)}$ and $Q_{m,n}^{(k)}$ to satisfy the equation
\begin{align}
&iP_{m,n}^{(k)}-Q_{m,n}^{(k)}:=\nonumber\\
&\frac{-\langle
N_{m,n}^{Im},\eta_{k}\rangle+i\langle N_{m,n}^{Re},\xi_{k}\rangle+i\displaystyle\sum_{|m|=|k|=|r|=1, k+r=n}\Upsilon_{m,k}[\langle r\cdot \eta,\eta_{k}\rangle-\langle r\cdot\xi,\xi_{k}\rangle]}{E_k-E(\lambda)\cdot (m-n)},\label{eq:PQkmn}
\end{align} here the denominator $E_k-E(\lambda)\cdot (m-n)$ stays away from zero by a uniform margin, by the same justification as in \eqref{eq:awayZero}, and the definition of $\Upsilon_{m,n}$ and the convention on $r$ is as in \eqref{eq:A1}, $r\cdot \eta:=\sum_{k}r_k\eta_k$ and $r\cdot \xi:=\sum_{k}r_k\xi_k$. Note that at this moment \eqref{eq:PQkmn} does not give unique solutions. This will be become clear in a moment.

For $|m|=2$ and $|n|=1,$ we define
\begin{align}
i P_{m,n}^{(k)}-Q_{m,n}^{(k)}=0.
\end{align}

After defining $iP_{m,n}^{(k)}-Q_{m,n}^{(k)}$ for $(|m|,\ |n|)=(1,2), \ (2,1)$ above, it is not hard to see that these together with the relations $P_{m,n}^{(k)}=\overline{P_{n,m}^{(k)}}$ and $Q_{m,n}^{(k)}=\overline{Q_{n,m}^{(k)}}$ determine unique solutions for the linear equations.

We continue to define $P^{(k)}_{m,n},\ Q^{(k)}_{m,n}$ for $|m|=|n|=1,$
\begin{align}
-i(m-n)\cdot E(\lambda)\ P^{(k)}_{m,n}-E_k(\lambda)Q_{m,n}^{(k)}:=&-\langle N_{m,n}^{Im},\eta_{k}\rangle,\nonumber\\
-i(m-n)\cdot E(\lambda)\ Q^{(k)}_{m,n}+E_k(\lambda)P_{m,n}^{(k)}:=&\langle N_{m,n}^{Re},\xi_{k}\rangle.
\end{align} The solutions are well defined and unique since the matrix 
\begin{align}\left(
\begin{array}{cc}
-i(m-n)\cdot E(\lambda)& -E_k(\lambda)  \\
 E_k(\lambda) & -i(m-n) \cdot E(\lambda) 
\end{array}\right)
\end{align} is uniformly invertible by the same arguments in showing the invertibility of the matrix in \eqref{Dmatrix}.

We complete defining all the relevant terms by requiring that
$$P_{m,n}^{(k)}:=\overline{P_{n,m}^{(k)}},\ Q_{m,n}^{(k)}:=\overline{Q_{n,m}^{(k)}}.$$

By now we have finished defining the polynomials $a_1,\ a_2,\ p_k,\ q_k,\ k=1,2,\cdots,N.$

Next we study the equation for $\dot{z}_k.$ By the definitions of coefficients $P_{m,n}^{(k)}, \ Q_{m,n}^{(k)}$
we removed the following terms from the $\dot{z}_k$-equation: $z^{m}\bar{z}^{n}$ if $|m|+|n|=2,3$ and $(|m|, \ |n|)\not=(2,1).$ The result is:
\begin{proposition}\label{Prop:ExplicitPolyno}
Define the polynomials $a_{1}(z,\bar{z}),\ a_{2}(z,\bar{z}),\ p_{n}(z,\bar{z}),\ q_{n}(z,\bar{z})$ as
above. Then, 
 \eqref{ExpanLambda}-\eqref{ExpanGamma} holds and
moreover for $k=1,2,\cdots, N$
\begin{align}\label{eq:ZNequation}
\partial_{t}z_{k}+iE_k(\lambda)z_{k}=&-\left< \sum_{|m|=2,\ |n|=1}JN_{m,n} z^{m}\bar{z}^{n},
\left(
\begin{array}{lll}
\eta_{k}\\
-i\xi_{k}
\end{array}\right)\right>\\
&+
\frac{1}{2}\Upsilon\displaystyle\sum_{m=1}^{N}z_{m}\left<\left(
\begin{array}{lll}
-i\eta_{m}\\ \xi_{m} \end{array}\right), 
\left(
\begin{array}{lll}
\eta_{k}\\
i\xi_{k}
\end{array}
\right)\right>+Remainder(t).\nn
\end{align} 
\end{proposition}
\begin{proof}

Recall the convention that $Remainder$ represents any quantity satisfying
\begin{align}
\lesssim |z(t)|^{4}+\|\langle
x\rangle^{-\nu}\vec{R}(t)\|_{2}^{2}+\|\vec{R}(t)\|_{\infty}^{2}
+|z(t)|\  \|\langle
x\rangle^{-\nu}\tilde{R}(t)\|_{2}.
\end{align} 

We start with casting the $\dot\lambda-$ and $\dot\gamma-$eqns in  
\eqref{eq:gamma}, \eqref{eq:lambda} into a matrix form
\begin{align}\label{eq:lambdagamma}
[Id+M(z,\vec{R},p,q)]\left(
\begin{array}{lll}
\dot\lambda\\
\dot\gamma-\Upsilon
\end{array}
\right)=\Omega+Remainder
\end{align} where, the vector $\Omega$ is defined as
\begin{align}
\Omega:=\left(
\begin{array}{lll}
-\frac{1}{\langle
\phi^{\lambda},\partial_{\lambda}\phi^{\lambda}\rangle}[\ \langle
ImN,\phi^{\lambda}\rangle+\frac{i}{2}\Upsilon\langle\ (z-\bar{z})\cdot\eta,\phi^{\lambda}\ \rangle\ ]-\partial_{t}a_{1}\\
\frac{1}{\langle\phi^{\lambda},\partial_{\lambda}\phi^{\lambda}\rangle}[\langle
ReN,\partial_{\lambda}\phi^{\lambda}\rangle-\frac{1}{2}\Upsilon\langle\
(z+\bar{z})\cdot\xi,\partial_{\lambda}\phi^{\lambda}\ \rangle]-\Upsilon-\partial_{t}a_{2}+a_{1}
\end{array}
\right)
\end{align}
the term $Remainder$ is produced by the term
$\frac{\Upsilon}{\langle\phi^{\lambda},\partial_{\lambda}\phi^{\lambda}\rangle}\left(
\begin{array}{lll}
-\langle R_{1},\partial_{\lambda}\phi^{\lambda}\rangle\\
\langle R_{2},\phi^{\lambda}\rangle
\end{array}
\right),$ $Id$ is the $2\times 2$ identity matrix,
$M(z,\vec{R},p,q)$ is a matrix depending on $z,\vec{R},p$ and $q$
and satisfying the estimate
\begin{equation}\label{eq:Mterm}
\|M(z,\vec{R},p,q)\|= \cO(|z|)+Remainder.
\end{equation}

The smallness of the matrix $M$ makes $\big[Id+M\big]^{-1}$ uniformly bounded,  hence by \eqref{eq:lambdagamma}
\begin{align}
|\dot\lambda|,\ |\dot\gamma-\Upsilon|\lesssim  |\Omega|+Remainder.\label{eq:LamGam}
\end{align}

Next we estimate $\Omega$, and start with casting it into a convenient form.

The purpose of defining $a_{1}$ and $a_{2}$ in \eqref{eq:pkmn} is to
remove the lower order terms, in $z$ and $\bar{z}$, from $\langle
ImN,\phi^{\lambda}\rangle-\frac{i}{2}\Upsilon\langle\ 
(z-\bar{z})\cdot\eta,\phi^{\lambda}\rangle$ and $\langle
ReN,\partial_{\lambda}\phi^{\lambda}\rangle+\frac{1}{2}\Upsilon\langle\ 
(z+\bar{z})\cdot\xi,\partial_{\lambda}\phi^{\lambda}\rangle$ to get 
\begin{align}
\Omega=D_{1}+D_{2}\label{eq:d1d2}
\end{align} with
$$D_{1}:=\frac{1}{\langle\phi^{\lambda},\partial_{\lambda}\phi^{\lambda}\rangle}\left(
\begin{array}{lll}
-\langle ImN-\displaystyle\sum_{|m|+|n|=2,3}N_{m,n}^{Im} z^{m}\bar{z}^{n},\phi^{\lambda}\rangle\\
\langle
ReN-\displaystyle\sum_{|m|+|n|=2,3}N_{m,n}^{Re} z^{m}\bar{z}^{n},\partial_{\lambda}\phi^{\lambda}\rangle
\end{array}
\right),$$ and
$$D_{2}:=-\displaystyle\sum_{|m|+|n|=2,3}\left(
\begin{array}{lll}
\partial_{t}(A_{m,n}^{(1)}z^m\bar{z}^{n})+iE(\lambda)\cdot (m-n)A_{m,n}^{(1)}z^m\bar{z}^{n}\\
\partial_{t}(A_{m,n}^{(2)}z^m\bar{z}^{n})+iE(\lambda)\cdot (m-n) A_{m,n}^{(2)}z^m\bar{z}^{n}
\end{array}
\right).$$ It is easy to see that
\begin{align}
D_{1}=Remainder.\label{eq:estD1}
\end{align} 

To control $D_2$ we use a preliminary estimate from the $z_n-$equation in \eqref{eq:z}
\begin{align}
\partial_{t}z_{n}+iE_n(\lambda)z_{n}=\cO(|z|^{2})+Remainder\label{eq:preli}
\end{align} to obtain 
\begin{align}
D_{2}=&-\dot\lambda\displaystyle\sum_{|m|+|n|=2,3}\left(
\begin{array}{lll}
\partial_{\lambda}A_{m,n}^{(1)}\ z^m\bar{z}^{n}\\
\partial_{\lambda}A_{m,n}^{(2)}\ z^m\bar{z}^{n}
\end{array}
\right)+\cO(|z|^{3})+Remainder\nn\\
=&\cO(|z|^{3})+Remainder\label{eq:estD2}
\end{align} here the term $\cO(|z|^3)$ is from the term $\cO(|z|^2)$ in \eqref{eq:preli}.

Collect the estimates above to obtain
\begin{align}
\dot\lambda,\ \dot\gamma-\Upsilon=&\cO(|z|^{3})+Remainder.\label{eq:LambdaGammaRough}
\end{align}
These estimates are still worse than the desired \eqref{ExpanLambda}-\eqref{ExpanGamma}. The reason is that their derivations depend on the non-optimal \eqref{eq:preli}. Next we improve it using \eqref{eq:LambdaGammaRough}.

Choose $p_{n}$ and $q_{n}$ as in
\eqref{eq:pq20and02} to remove the following lower order terms: $z^{m}\bar{z}^{n}$ satisfying $|m|+|n|=2,3$ and $(|m|,\ |n|)\not=(2,1),$ to obtain
\begin{align}\label{eq:ZnRough}
\partial_{t}z_{n}+iE_n(\lambda)z_{n}&=-\left\langle \sum_{|m|=2,\ |n|=1}JN_{m,n}z^{m}\bar{z}^{n}
 +\frac{1}{2}\Upsilon\left(
\begin{array}{lll}
iz\cdot\eta\\
z\cdot\xi
\end{array}
\right), \left(
\begin{array}{lll}
\eta_{n}\\
-i\xi_{n}
\end{array}
\right)\right\rangle\nn\\
&+ D_{3}(n)+Remainder
\end{align}
where $D_{3}(n)$ is defined as
\begin{align}
D_{3}(n):=&-\sum_{|k|+|l|=2,3}[\partial_{t}(P_{k,l}^{(n)}z^{k}\bar{z}^{l})+i(k-l)\cdot E(\lambda)P_{k,l}^{(n)}z^{k}\bar{z}^{l}]\nonumber\\
&-i\sum_{|k|+|l|=2,3}[\partial_{t}(Q_{k,l}^{(n)}z^{k}\bar{z}^{l})+i(k-l)\cdot E(\lambda)Q_{k,l}^{(n)}z^{k}\bar{z}^{l}]\nonumber\\
=&-\dot\lambda \ \sum_{|k|+|l|=2,3}[\partial_{\lambda}P_{k,l}^{(n)}-\partial_{\lambda}Q_{k,l}^{(n)}]z^{k}\bar{z}^{l}\nonumber\\
&-\sum_{|k|+|l|=2,3}[P_{k,l}^{(n)}-Q_{k,l}^{(n)}]\ [\partial_{t}(z^{k}\bar{z}^{l})+i(k-l)\cdot E(\lambda)z^{k}\bar{z}^{l}]\nonumber\\
=& \Gamma_1+\Gamma_2
\end{align} and $\Gamma_1$ and $\Gamma_2$ naturally defined.

For $\Gamma_1$ the estimate for $\dot\lambda$ in \eqref{eq:LambdaGammaRough} implies that
\begin{align}
\Gamma_1=Remainder.\label{eq:estGamma1}
\end{align} 

For $\Gamma_2$, the preliminary estimate in \eqref{eq:preli} implies
$$\Gamma_2=\cO(|z|^{3})+Remainder.$$ This, in turn, implies an estimate better than \eqref{eq:preli}
\begin{align}
\partial_t z_{n}+iE_n(\lambda)z_{n}=\cO(|z|^{3})+Remainder.\label{eq:betterZn}
\end{align} Sine the estimates derived for $\Gamma_2$ depends on non-optimal \eqref{eq:preli}, this optimal one enables us to find
\begin{align}
\Gamma_2=Remainder.
\end{align} This together with $\Gamma_1=Remainder$ in \eqref{eq:estGamma1} implies $$D_3(n)=\Gamma_1+\Gamma_2=Remainder.$$ Put this into the $\partial_{t}z_n$-eqn 
in \eqref{eq:ZnRough} to obtain the desired estimate \eqref{eq:ZNequation}.

\eqref{eq:betterZn} also helps us to improve the estimate \eqref{eq:estD2} for $D_2$
\begin{align}
D_2=Remainder.
\end{align} This together with \eqref{eq:estD1}, \eqref{eq:d1d2} implies for $\Omega$ in \eqref{eq:d1d2}
$$\Omega=D_1+D_2=Remainder.$$
Put this into \eqref{eq:LamGam} to obtain
the desired estimates \eqref{ExpanLambda}, \eqref{ExpanGamma} for $\dot\lambda$ and $\dot\gamma-\Upsilon.$

\end{proof}

The following result has been applied in \eqref{defA11} to show that the numerator is proportional to the denominator. Similar result can be found in \cite{GaWe}.
\begin{lemma}\label{coordinate-lemma}
For $|m|=|n|=1,$ and $m,\ n\in (\mathbb{Z}^{+}\cup\{0\})^{N}$
\begin{align}
\langle N_{m,n}^{Im},\ \phi^{\lambda}\rangle=\frac{1}{4i}(n-m)\cdot E(\lambda)\ [\langle m\cdot\xi,\ n\cdot \xi\rangle+\langle m\cdot \eta, \ n\cdot \eta\rangle].
\label{eq:UnexpectedFact}
\end{align}
\end{lemma}
\begin{proof}
We start with deriving an expression for $N_{m,n}^{Im},\ |m|=|n|=1.$

The explicit form of $\displaystyle\sum_{|m|+|n|=2}JN_{m,n}z^{m}\bar{z}^{n}$ in
\eqref{eq:SecondOrderTerm} implies that 
\begin{align*}
\sum_{|m|+|n|=2}N^{Im}_{m,n}z^{m}\bar{z}^{n}=&2\phi^{\lambda}A_{1}B_{1}\\
=&\frac{1}{2i}\phi^{\lambda}(\displaystyle\sum_{n=1}^{N}z_{n}\xi_{n}+\sum_{n=1}^{N}\bar{z}_{n}\xi_{n})
(\sum_{m=1}^{N}z_{m}\eta_{m}-\sum_{m=1}^{N}\bar{z}_{m}\eta_{m}).
\end{align*} 
Take the relevant terms to obtain
\begin{align}
N^{Im}_{m,n}=\frac{1}{2i}\phi^{\lambda}(\xi_{n}\eta_{m}-\xi_{m}\eta_{n}),
\end{align} here we used the notation
\begin{align}
\xi_n=n\cdot \xi,\ \eta_n=n\cdot \eta,\ \text{for}\ n\in (\mathbb{Z}^{+}\cup\{0\})^{N}\ \text{and}\ |n|=1.\label{eq:notationXiEta}
\end{align}

Hence the left hand side of \eqref{eq:UnexpectedFact} takes a new form,
\begin{align*}
\langle
N^{Im}_{m,n},\phi^{\lambda}\rangle=&\frac{1}{2i}\int
(\phi^{\lambda})^{2}(\xi_{n}\eta_{m}-\xi_{m}\eta_{n})\\
=&\frac{1}{4i} \big[\langle [L_{+}(\lambda)-L_{-}(\lambda)]\xi_n, \eta_{m}\rangle-\langle [L_{+}(\lambda)-L_{-}(\lambda)]\xi_m, \eta_{n}\rangle\big],
\end{align*} where we used the fact that $L_{+}(\lambda)-L_{-}(\lambda)=2(\phi^{\lambda})^2$, see \eqref{eq:defLLambda}.
Use the facts $L_{+}(\lambda)$ and $L_{-}(\lambda)$ are self-adjoint, and $$L_{-}(\lambda)\eta_{n}\ =\ E_n(\lambda)\xi_{n},\
\text{and}\ L_{+}(\lambda)\xi_{n}\ =\ E_n(\lambda)\eta_{n},\
n=1,2,\cdot\cdot\cdot,N,$$ in \eqref{eq:eigenf}, and hence the desired result, recall the notations in \eqref{eq:notationXiEta},
\begin{align}
\langle
N^{Im}_{m,n},\phi^{\lambda}\rangle=&\frac{1}{4i}  (E_n(\lambda)-E_m(\lambda))[\langle \xi_m,\ \xi_n\rangle+\langle \eta_m,\ \eta_n\rangle]\nn\\
=&\frac{1}{4i}(n-m)\cdot E(\lambda)\ [\langle m\cdot\xi,\ n\cdot \xi\rangle+\langle m\cdot \eta, \ n\cdot \eta\rangle]
\end{align}
\end{proof}

\section{Proof of the Normal Form Equation \eqref{eq:detailedDescription}}\label{sec:NFT}

We expand the first two terms on the right hand side of the equations for $z_{n}$ in
\eqref{eq:ZNequation} to obtain
\begin{align}
&-\left< \sum_{|m|=2, |n|=1}JN_{m,n} z^{m}\bar{z}^{n},
\left(
\begin{array}{lll}
\eta_{k}\\
-i\xi_{k}
\end{array}\right)\right>
+
\frac{1}{2}\Upsilon\displaystyle\sum_{m=1}^{N}z_{m}\left<\left(
\begin{array}{lll}
-i\eta_{m}\\ \xi_{m} \end{array}\right), 
\left(
\begin{array}{lll}
\eta_{k}\\
i\xi_{k}
\end{array}
\right)\right>\nn\\
&=\sum_{l=1}^{5}\Theta_{l}(k),
\end{align} where, recall the definitions of $JN_{m,n},\ |m|+|n|=3,$ from
\eqref{eq:JNm+n=3}, $\Theta_1(k)$ is defined as
\begin{align}
\Theta_{1}(k) :=-\left< \overline{X_1}R_{2,0},\left(
\begin{array}{lll}
\eta_{k}\\
-i\xi_{k}
\end{array}
\right)\right>=-\left< \sum_{|m|=2}R_{m,0}z^{m},\ (\overline{X_1})^* \left(
\begin{array}{lll}
\eta_{k}\\
-i\xi_{k}
\end{array}
\right)\right>,\label{eq:Theta1k}
\end{align}
where, recall the definition of $2\times 2$ matrix $X$ in \eqref{eq:XDef} and we divide it into two terms
$X=X_{1}+\overline{X_{1}}$ with $X_1$ defined as
\begin{align}
X_{1}:=\left(
\begin{array}{lll}
-i\phi^{\lambda}\ z\cdot\eta,& \phi^{\lambda}\ z\cdot\xi\\
-3 \phi^{\lambda } z\cdot\xi,&
i\phi^{\lambda}\ z\cdot\eta
\end{array}
\right),\label{eq:defX1}
\end{align}

\begin{align}
 \Theta_{2}(k):=&-\left< X_{1}\left(
\begin{array}{lll}
\displaystyle\sum_{n=1}^{N}\sum_{|m|=|l|=1}P_{m,l}^{(n)} z^{m}\bar{z}^{l}\xi_{n}\\
\displaystyle\sum_{n=1}^{N}\sum_{|m|=|l|=1}Q_{m,l}^{(n)}z^{m}\bar{z}^{l}\eta_{n}
\end{array}
\right),\ \left(
\begin{array}{lll}
\eta_{k}\\
-i\xi_{k}
\end{array}
\right)\right>\nn\\
&-\left< \overline{X_{1}}\left(
\begin{array}{lll}
\displaystyle\sum_{|m|=2}z^{m}[\sum_{n=1}^{N}P_{m,0}^{(n)}\xi_{n}+A_{m,0}^{(1)}\partial_{\lambda}\phi^{\lambda}]\\
\displaystyle\sum_{|m|=2}z^{m}[\sum_{n=1}^{N}Q_{m,0}^{(n)}\eta_{n}+A_{m,0}^{(2)}\phi^{\lambda}]
\end{array}
\right), \left(
\begin{array}{lll}
\eta_{k}\\
-i\xi_{k}
\end{array}
\right)\right>,\label{smallDi}
\end{align}

\begin{align*}
\Theta_{3}(k):=&-\frac{1}{8}\left\langle\ 
 ((z\cdot\xi)^{2}-(z\cdot\eta)^{2})\left(
\begin{array}{lll}
i\bar{z}\cdot\eta\\ -\bar{z}\cdot\xi
\end{array}
\right), \left(
\begin{array}{lll}
\eta_{k}\\
-i\xi_{k}
\end{array} \right)\ \right\rangle\\
&+\frac{1}{4}\left\langle\ 
(|z\cdot\xi|^2+|z\cdot\eta|^2)\left(
\begin{array}{lll}
i z\cdot\eta\\ z\cdot\xi
\end{array} \right), \left(
\begin{array}{lll}
\eta_{k}\\ 
-i\xi_{k}
\end{array} \right)\right\rangle,
\end{align*}
\begin{align*}
\Theta_{4}(k):=\frac{1}{2}\ \Upsilon\left\langle\left(
\begin{array}{lll}
-i z\cdot\eta\\
z\cdot\xi
\end{array}
\right), \left(
\begin{array}{lll}
\eta_{k}\\
i\xi_{k}
\end{array}
\right)\right\rangle,
\nn\end{align*}
\begin{align*}
\Theta_{5}(k):=-\left\langle \sum_{|m|=|n|=1}R_{m,n}z^{m}\bar{z}^{n},\ X^{*}_{1}\left(
\begin{array}{lll}
\eta_{k}\\
-i\xi_{k}
\end{array}
\right)\right
\rangle,
\end{align*}
where, recall the definition of real function $\Upsilon$ from
\eqref{ExpanGamma}.

The result is the following
\begin{proposition}\label{prop:smallDivisor}
$Re\ \displaystyle\sum_{k=1}^{N}\bar{z}_k \Theta_1(k)$ can be decomposed into three terms,
\begin{align}
Re\ \sum_{k=1}^{N}\bar{z}_k \Theta_1(k)
=&- C (e_0-\lambda)\ \Gamma(z,\bar{z})\nn\\
&+(e_0-\lambda) \sum_{|m|=|n|=2}C_{m,n}(\lambda)\ E(\lambda)\cdot(m-n)\ z^{m}\bar{z}^{n}\label{eq:ReThet1}\\
&+|e_0-\lambda|^2\ \cO(|z|^4)  \nn
\end{align} where $C$ is a positive constant, $\Gamma(z,\bar{z})$ is the positive term defined in Fermi-Golden-rule \eqref{Gammadef}, $C_{m,n}(\lambda)$ are uniformly bounded constants and recall that $e_0-\lambda>0,$ in Lemma \ref{LM:groundNon}.

For $Re\ \displaystyle\sum_{k=1}^{N}\bar{z}_k \Theta_2(k)$, there exist some constants $D_{m,n}(\lambda)$ such that
\begin{align}
Re\ \sum_{k=1}^{N}\bar{z}_k \Theta_2(k)
=&(e_0-\lambda) \sum_{|m|=|n|=2}D_{m,n}(\lambda)\ E(\lambda)\cdot(m-n)\ z^{m}\bar{z}^{n};\label{eq:ReThet2}
\end{align}

For $n=3,4,5,$ we have
\begin{align}
Re\ \sum_{k=1}^{N}\bar{z}_k \Theta_n (k)=0.\label{eq:ReThet3}
\end{align}
\end{proposition}
The proposition will be proved in subsequent subsections.

Next we prove the desired result \eqref{eq:detailedDescription}.\\
\textbf{Proof of \eqref{eq:detailedDescription}}
By the estimates above we have that
\begin{align}
\partial_{t}|z|^2=&-2C(e_0-\lambda)\ \Gamma(z,\bar{z})+(e_0-\lambda) \sum_{|m|=|n|=2} C_{m,n} \ E(\lambda)\cdot (m-n)z^{m}\bar{z}^n\nn\\
&+\cO((e_0-\lambda)^2|z|^4)+|z|Remainer,
\end{align} with $C_{m,n}$ being some constant. Here we can take $\displaystyle\sum_{|m|=|n|=2} C_{m,n} \ E(\lambda)\cdot (m-n)z^{m}\bar{z}^n$ to be real since it is in the equation for real parameter $|z|^2$ and $C\Gamma(z,\bar{z})$ is real. This implies $$\sum_{|m|=|n|=2} C_{m,n} \ E(\lambda)\cdot (m-n)z^{m}\bar{z}^n=\sum_{|m|=|n|=2} \overline{C_{m,n}} \ E(\lambda)\cdot (m-n)\bar{z}^{m} z^n,$$ and hence forces 
\begin{align}
C_{m,n}=-\overline{C_{n,m}}.\label{eq:realIma}
\end{align}
By observing that $$E(\lambda)\cdot (m-n)z^{m}\bar{z}^n=i\partial_t z^{m}\bar{z}^n+|z|Remainder,$$ the fact that $(e_0-\lambda) \displaystyle\sum_{|m|=|n|=2}i C_{m,n}  z^{m}\bar{z}^{n}$ is real implied by \eqref{eq:realIma}, we can define a new nonnegative parameter $\tilde{z}$ satisfying $\tilde{z}^2=|z|^2-(e_0-\lambda) \displaystyle\sum_{|m|=|n|=2}i C_{m,n}  z^{m}\bar{z}^{n}$ such that
\begin{align}
\partial_{t} \tilde{z}^2= -2C(e_0-\lambda)\ \Gamma(z,\bar{z})+\cO((e_0-\lambda)^2|z|^4)+|z|Remainer
\end{align}
which is the desired estimate \eqref{eq:detailedDescription}.

\begin{flushright}
$\square$
\end{flushright}

Next we prove Proposition \ref{prop:smallDivisor}. In the proof the following results, from \cite{GaWe2011}, will be used. Recall that the function $\phi$ is the ground state for $-\Delta+V$ with eigenvalue $-e_0$, the functions $\xi^{lin}_{k}$, $k=1,2,\cdots,N,$ are neutral modes with eigenvalues $-e_k.$
\begin{lemma}
There exist constants $C_0,\ C_1\in\mathbb{R}$ such that in the space $\langle x\rangle^{-4}\mathcal{H}^{2}$
\begin{align}
\phi^{\lambda}=&C_0 (e_0-\lambda)^{\frac{1}{2}} \phi+\cO(|e_0-\lambda|^{\frac{3}{2}}),\nn\\
\partial_{\lambda}\phi^{\lambda}=&C_1 (e_0-\lambda)^{-\frac{1}{2}}\phi+\cO(|e_0-\lambda|^{\frac{1}{2}}),\label{eq:LambdaPhi2}\\
&\frac{1}{\langle \phi^{\lambda},
\partial_{\lambda}\phi^{\lambda}\rangle}\lesssim 1.\nn
\end{align}
For the neutral modes we have
\begin{align}\label{eq:asympto}
\|\langle x\rangle^{4}(\eta_{m}-\xi_{m}^{lin})\|_{\mathcal{H}^2},\ \|\langle x\rangle^{4}(\xi_{m}-\xi_{m}^{lin})\|_{\mathcal{H}^2},\ \|\langle x\rangle^{4}(\xi_{m}-\eta_{m})\|_{\mathcal{H}^2}=\cO(|e_0-\lambda|).
\end{align} Recall $P_{c}^{lin}$ is the orthogonal project onto the essential spectrum of $-\Delta+V$, and $P_{c}^{\lambda}$ of \eqref{Pcdef} is the Riesz projection onto the essential spectrum of $L(\lambda)$
\begin{align}\label{eq:projection}
P_{c}^{\lambda}=P_{c}^{lin}\left(
\begin{array}{ll}
1&0\\
0&1
\end{array}
\right)+\cO(|e_0-\lambda|).
\end{align} 
\end{lemma}
\subsection{Proof of \eqref{eq:ReThet1}}
We start from the definition of $\Theta_1(k)$ in \eqref{eq:Theta1k}.
Compute directly to obtain
\begin{align*}
(\overline{X_1})^* \left(
\begin{array}{lll}
\eta_{k}\\
-i\xi_{k}
\end{array}
\right)=\left(
\begin{array}{ll}
-i\phi^{\lambda} z\cdot \eta\ \eta_k+3i \phi^{\lambda} z\cdot \xi\ \xi_k\\
\phi^{\lambda} z\cdot \xi\ \eta_k+\phi^{\lambda} z\cdot \eta\ \xi_k
\end{array}
\right).
\end{align*}
Hence 
\begin{align}
Re \sum_{k}\bar{z}_k\Theta_{1}(k)=-\left< \sum_{|m|=2}R_{m,0}z^{m},\  \left(
\begin{array}{lll}
-i\phi^{\lambda} (z\cdot \eta)^2+3i \phi^{\lambda} (z\cdot \xi)^2\\
2\phi^{\lambda} z\cdot \xi\ z\cdot \eta
\end{array}
\right)\right>.\label{eq:theta1k}
\end{align}

We extract its main part by define 
\begin{align}
D:=Re \sum_{|m|=2}&\big\langle z^{m}\big[(-\Delta+V+\lambda)J+i E(\lambda)\cdot m-0\big]^{-1} P_c \phi  (\xi^{lin})^m \left(
\begin{array}{lll}
-i\\
1
\end{array}
\right),\ \times\nn\\
&\phi (z\cdot \xi^{lin})^2\left(
\begin{array}{lll}
i\\
1
\end{array}
\right)
\big\rangle.\label{eq:defD}
\end{align}
Here $(\xi^{lin})^{m}$ is defined as, for $m=(m_1,\ m_2,\ \cdots, \ m_N)\in (\mathbb{Z}^{+}\cup\{0\})^{N}$,
\begin{align*}
(\xi^{lin})^{m}:=(\xi_1^{lin})^{m_1}(\xi_2^{lin})^{m_2}\cdots (\xi_N^{lin})^{m_N}.
\end{align*}

The result is
\begin{lemma}\label{LM:simplify}
\begin{align}
Re \sum_{k=1}^{N}\bar{z}_k\Theta_{1}(k)&=-C(e_0-\lambda ) \ D+\cO(|e_0-\lambda |^2 |z|^4).
\end{align}
\end{lemma}
The lemma will be proved in Part \ref{Pa:simplify} below.

Next we study the term $D$. To facilitate our estimate we diagonalize the matrix operator in \eqref{eq:defD}. Define a unitary matrix $U$ by
 \begin{align}
 U\ :=\ \frac{1}{\sqrt{2}}\left(
\begin{array}{lll}
1&i\\ 
i&1 
\end{array}
\right),\label{Udef}
\end{align} then we have that 
$$J=i U \sigma_3 U^*,$$ with $\sigma_3$ being the third Pauli matrix.

Inset the identity $UU^*=U^* U=Id$ into appropriate places in the inner product of $D$ to obtain, recall the convention that $\langle f,\ g\rangle=\int f(x)\bar{g}(x)\ dx,$
\begin{align}
D
=&2\sum_{|m|=|n|=2}Re \langle
\big[i(-\Delta+V+\lambda- E(\lambda)\cdot m)+0\big]^{-1} P_c\phi (\xi^{lin})^{m},\ \phi (\xi^{lin})^{n}
\rangle \ z^{m}\bar{z}^{n}\nn\\
=&2 \sum_{|m|=|n|=2}Im\langle\big[-\Delta+V+\lambda- E(\lambda)\cdot m-i0\big]^{-1} P_c\phi (\xi^{lin})^{m},\ \phi (\xi^{lin})^{n}\rangle\ z^{m}\bar{z}^{n}.\label{eq:secondRes}
\end{align}

To cast the expression into a convenient form, we use the following two simple facts, for any functions $f,\ g$ and real constant $h$, $Im\langle f,\ g\rangle=\frac{1}{2i}[\langle f,\ g\rangle-\langle g,\ f\rangle]$ and $\langle f,\ (-\Delta+V-h-i0)^{-1}g\rangle=\langle (-\Delta+V-h+i0)^{-1}f,\ g\rangle.$ Compute directly to obtain
\begin{align}
D=&\frac{1}{i} \sum_{|m|=|n|=2} \langle L(m,n)\ P_c\phi (\xi^{lin})^{m},\ \phi (\xi^{lin})^{n}\rangle\ z^{m}\bar{z}^{n}.\label{express}
\end{align}
with $L(m,n)$ being a linear operator defined as
$$L(m,n):=
[-\Delta+V+\lambda- E(\lambda)\cdot m-i0]^{-1}-[-\Delta+V+\lambda- E(\lambda)\cdot n+i0]^{-1}
$$

In studying \eqref{express}, the main tool is a well known fact that, see e.g. \cite{RSIII}, for any constant $h>0,$
\begin{align}
&\langle \frac{1}{i} \big[ [-\Delta+V-h^2-i0]^{-1}-[-\Delta+V-h^2+i0]^{-1}\big]\ P_c f,\ g\rangle\nn\\
=& C h\int_{\mathbb{S}^2} \widehat{f}(h\sigma) \ \overline{\widehat{g}}(h\sigma)\ d\sigma  \label{eq:measure}
\end{align} where $C$ is a positive constant, and the complex function $\widehat{f}$ is defined as,
\begin{align} 
\widehat{f}(k):= \int_{\mathbb{R}^3} f(x)\ e(x, k)\ dx,
\end{align} hence $\widehat{f}(h\sigma)$ is $\widehat{f}$ restricted to the sphere $|k|=h\sigma,\ \sigma\in \mathbb{S}^2.$ Here the complex function $e:\ \mathbb{R}^3\times \mathbb{R}^3\rightarrow \mathbb{C}$ is defined as $$e(x,k):=[1+(-\Delta-|k|^2-i0)^{-1}V(x)]^{-1}\ e^{ix\cdot k}.$$ 

We continue to study \eqref{express}.
For the easiest cases
$
 E(\lambda)\cdot m= E(\lambda)\cdot n, 
$ apply \eqref{eq:measure} directly to obtain
\begin{align}
\frac{1}{i}  \langle L(m,n) P_c\phi (\xi^{lin})^{m},\ \phi (\xi^{lin})^{n}\rangle
=C C_m \int_{\mathbb{S}^2} \widehat{\phi(\xi^{lin})^m}(C_m\sigma) \ \overline{\widehat{\phi(\xi^{lin})^n}}(C_m\sigma)\ d\sigma \label{eq:directFa}
\end{align} with $C_{m}\in \mathbb{C}$ defined as $$C_m:=\sqrt{E(\lambda)\cdot m-\lambda},$$ here $E(\lambda)\cdot m-\lambda$ is positive by the conditions that $2e_k<e_0$ and $E_k(\lambda)\approx e_0-e_k ,\ k=1,2,\cdots, N,$ and $\lambda\approx e_0.$ Recall that $m\in (\mathbb{Z}^{+}\cup\{0\})^{N}$ and $|m|=2.$

For the cases $ E(\lambda)\cdot m\not= E(\lambda)\cdot n, $ we claim, for some constant $C_{m,n}$
\begin{align}
\frac{1}{i}  \langle L(m,n) P_c\phi (\xi^{lin})^{m},\ \phi (\xi^{lin})^{n}\rangle
=&C C_m^{\frac{1}{2}} C_n^{\frac{1}{2}} \int_{\mathbb{S}^2} \widehat{\phi(\xi^{lin})^m}(C_m\sigma) \ \overline{\widehat{\phi(\xi^{lin})^n}}(C_n\sigma)\ d\sigma\nn\\
&+C_{m,n} E(\lambda)\cdot (m-n).\label{eq:claim88}
\end{align} 
If the claim holds, this together with \eqref{eq:directFa}, Lemma \ref{LM:simplify} and the fact $E_k(\lambda)=e_0-e_k+\cO(e_0-\lambda)$ by \eqref{eq:Enen}, implies the desired result \eqref{eq:ReThet1}.

What is left is to prove the claim \eqref{eq:claim88}. We start with decomposing 
the left hand side into two parts
\begin{align}
\frac{1}{i}  \langle L(m,n) P_c\phi (\xi^{lin})^{m},\ \phi (\xi^{lin})^{n}\rangle=A+B
\end{align} with
\begin{align*}
A:=&\frac{1}{i}  \langle L(m,m) P_c\phi (\xi^{lin})^{m},\ \phi (\xi^{lin})^{n}\rangle,\\
B:=&\frac{1}{i}  \langle [L(m,n)-L(m,m)] P_c\phi (\xi^{lin})^{m},\ \phi (\xi^{lin})^{n}\rangle.
\end{align*} By \eqref{eq:measure}, it is easy to see that 
\begin{align}
A=&C C_m \int_{\mathbb{S}^2} \widehat{\phi(\xi^{lin})^m}(C_m\sigma) \ \overline{\widehat{\phi(\xi^{lin})^n}}(C_m\sigma)\ d\sigma\\
=&C C_m^{\frac{1}{2}} C_n^{\frac{1}{2}} \int_{\mathbb{S}^2} \widehat{\phi(\xi^{lin})^m}(C_m\sigma) \ \overline{\widehat{\phi(\xi^{lin})^n}}(C_n\sigma)\ d\sigma+C_{m,n} E(\lambda)\cdot (m-n),\nn
\end{align} where $C_{m,n}$ is a constant, and in the second step we use that the functions $\widehat{\phi(\xi^{lin})^m}$ and the scalar $C_m$ depend smoothly on $E(\lambda)\cdot m$ and $E(\lambda)\cdot n$. 

For $B$, it is easy to see that for some constant $D_{m,n},$ $$B=D_{m,n} E(\lambda)\cdot (m-n).$$

Collecting the estimates above, we prove the claim \eqref{eq:claim88}.

Hence the proof is complete.

\subsubsection{Proof of Lemma \ref{LM:simplify}}\label{Pa:simplify}
We rewrite the expression in \eqref{eq:theta1k} as
\begin{align}
Re\sum_{k=1}^{N}\bar{z}_k \Theta_1(k)=-\left<
\sum_{|m|=2} R_{m,0}z^{m},\ A
\right>\label{eq:Aori}
\end{align} with the vector function $A$ defined as
\begin{align*}
A:=\left(
\begin{array}{lll}
-i\phi^{\lambda} (z\cdot \eta)^2+3i \phi^{\lambda} (z\cdot \xi)^2\\
2\phi^{\lambda} z\cdot \xi\ z\cdot \eta
\end{array}
\right).
\end{align*}

Apply the estimates of $\phi^{\lambda},$ $\xi$ and $\eta$ in 
\eqref{eq:LambdaPhi2} and \eqref{eq:projection} to obtain
\begin{align}
A=C_1 (e_0-\lambda)^{\frac{1}{2}}\phi (z\cdot \xi^{lin})^2\left(
\begin{array}{lll}
i \\
1
\end{array}
\right)+\cO((e_0-\lambda)^{\frac{3}{2}}|z|^2),\label{eq:expressA}
\end{align} here the expansion is in the space $\langle x\rangle^{-4}\mathcal{L}^2$ for some $\epsilon_0>0,$ $C_1>0$ is a constant.

Now we turn to $R_{m,0}$, which is defined as
\begin{align*}
R_{m,0}=\big[ L(\lambda)+iE(\lambda)\cdot m-0\big]^{-1} P_c^{\lambda} JN_{m,0};
\end{align*} and for $JN_{m,0}$ we have, from \eqref{eq:SecondOrderTerm},
\begin{align}
\sum_{|m|=2}JN_{m,0} z^{m}=\sum_{|m|=2}\left(
\begin{array}{lll}
ImN_{m,0}\\
-ReN_{m,0}
\end{array}
\right)z^{m}=\frac{1}{4} \left(
\begin{array}{cc}
-2i \phi^{\lambda} z\cdot \xi \ z\cdot \eta\\
-3\phi^{\lambda} (z\cdot \xi)^2+\phi^{\lambda} (z\cdot \eta)^2
\end{array}
\right)\label{eq:JNm0}
\end{align}

Now we extract the main part of $R_{m,0}$ and $JN_{m,0}$ by applying the estimates of $\phi^{\lambda},$ $\xi$ and $\eta$ and $P_c$ in 
\eqref{eq:LambdaPhi2} and \eqref{eq:projection}, and use that 
\begin{align*}
\big[ L(\lambda)+iE(\lambda)\cdot m-0\big]^{-1}=\big[(-\Delta+V+\lambda)J+i E(\lambda)\cdot m-0\big]^{-1}+\cO(e_0-\lambda)
\end{align*}
to obtain
\begin{align}
R_{m,0}
=& C_2 (e_0-\lambda)^{\frac{1}{2}}\ z^{m}\big[(-\Delta+V+\lambda)J+i E(\lambda)\cdot m-0\big]^{-1} P_c \phi  (\xi^{lin})^m \left(
\begin{array}{lll}
-i\\
1
\end{array}
\right)\nn\\
&+\cO((e_0-\lambda)^{\frac{3}{2}}) z^{m}
\label{eq:rmo}
\end{align}
where $C_2>0$ is a constant, the expansion is in the space $\langle x\rangle^{4}\mathcal{L}^2$.

Put this and \eqref{eq:expressA} into \eqref{eq:Aori} to obtain the desired result.

\begin{flushright}
$\square$
\end{flushright}

\subsection{Proof of \eqref{eq:ReThet2}}
To illustrate the ideas we consider part of it, namely
\begin{align}
\sum_{k}\bar{z}_k\tilde\Theta_2(k):=&\left< \overline{X_{1}}\left(
\begin{array}{lll}
\displaystyle\sum_{|m|=2}z^{m}\sum_{n=1}^{N}P_{m,0}^{(n)}\xi_{n}\\
\displaystyle\sum_{|m|=2}z^{m}\sum_{n=1}^{N}Q_{m,0}^{(n)}\eta_{n}
\end{array}
\right), \ 
\left(
\begin{array}{lll}
z\cdot \eta\\
-iz\cdot \xi
\end{array}
\right)\right>\nonumber.
\end{align} The other part is similar, hence omitted.
Compute directly to obtain
\begin{align}
&\sum_{k}\bar{z}_k\tilde\Theta_2(k)\nn\\
=
&\left< \left(
\begin{array}{lll}
\displaystyle\sum_{|m|=2}z^{m}\sum_{n=1}^{N}P_{m,0}^{(n)}\xi_{n}\\
\displaystyle\sum_{|m|=2}z^{m}\sum_{n=1}^{N}Q_{m,0}^{(n)}\eta_{n}
\end{array}
\right), \ 
(\overline{X_{1}})^*\left(
\begin{array}{lll}
z\cdot \eta\\
-iz\cdot \xi
\end{array}
\right)\right>\nonumber\\
=&4\left< \left(
\begin{array}{lll}
\displaystyle\sum_{|m|=2}z^{m}\sum_{n=1}^{N}P_{m,0}^{(n)}\xi_{n}\\
\displaystyle\sum_{|m|=2}z^{m}\sum_{n=1}^{N}Q_{m,0}^{(n)}\eta_{n}
\end{array}
\right), \ 
i\left(
\begin{array}{lll}
\displaystyle\sum_{|l|=2}ReN_{l,0}z^{l}\\
\displaystyle\sum_{|l|=2}ImN_{l,0}z^{l}
\end{array}
\right)\right>\nonumber\\
=& -4i \sum_{n=1}^{N}\sum_{|m|=2}\sum_{|l|=2} [ P_{m,0}^{(n)} z^{m} \ \langle \xi_n,\  ReN_{l,0}z^{l}\rangle+Q_{m,0}^{(n)} z^{m} \ \langle \eta_n,\  ImN_{l,0}z^{l}\rangle ]
\end{align} 
where, in the second last step we used the definition of $ReN_{l,0}$ and $ImN_{l,0}$, $|l|=2,$ in \eqref{eq:JNm0}. Next, relate the definitions of $P_{m,0}^{(n)}$ and $Q_{m,0}^{(k)}$ in \eqref{eq:pq20and02} to $\langle \eta_n,\ ImN_{m,0}\rangle$ and $\langle \xi_n,\ ReN_{m,0}\rangle$, and then take the real part for $\sum_{k}\bar{z}_k\tilde\Theta_2(k)$ to obtain
\begin{align}
Re\ \sum_{k}\bar{z}_k\tilde\Theta_2(k)
=& 4\sum_{n=1}^{N} \sum_{|m|=|l|=2}Re\big[ [P_{m,0}^{(n)}   \overline{Q_{l,0}^{(n)}} -Q_{m,0}^{(n)}\  \overline{P_{l,0}^{(n)}} ]\ l\cdot E(\lambda)\ z^{m}\bar{z}^{l}\big]\nonumber\\
=& 2\sum_{n=1}^{N} \sum_{|m|=|l|=2}Re\big[ [P_{m,0}^{(n)}   \overline{Q_{l,0}^{(n)}} -Q_{m,0}^{(n)}\  \overline{P_{l,0}^{(n)}} ]\ l\cdot E(\lambda)\ z^{m}\bar{z}^{l}\big]\nonumber\\
&+2\sum_{n=1}^{N} \sum_{|m|=|l|=2}Re\big[ [\overline{P_{m,0}^{(n)}}   Q_{l,0}^{(n)} -\overline{Q_{m,0}^{(n)}}\  P_{l,0}^{(n)} ]\ l\cdot E(\lambda)\ \bar{z}^{m} z^{l}\big]\nonumber\\
=&2\sum_{n=1}^{N} \sum_{|m|=|l|=2}Re\big[ [P_{m,0}^{(n)}   \overline{Q_{l,0}^{(n)}}-Q_{m,0}^{(n)}\  \overline{P_{l,0}^{(n)}} ]\ (l-m)\cdot E(\lambda)\ z^{m}\bar{z}^{l} \big]\nn\\
=&\sum_{|m|=|l|=2} C_{m,l}\ (l-m)\cdot E(\lambda) z^{m}\bar{z}^{l}\label{TildeTheta}
\end{align} where, in the second step we used a simple fact that $Re\ a=Re\ \bar{a}$, for any parameter $a$, for the second term in the second step, we interchange the indices $m$ and $l$ to obtain the third step, and in the last step $C_{m,l}$ are constants naturally defined.

The proof is complete by observing that \eqref{TildeTheta} is the desired result.

\subsection{Proof of \eqref{eq:ReThet3}}
It is easy to see that for $j=3,4,$ $ \displaystyle\sum_{k=1}^{N}\bar{z}_{k}\Theta_{j}(k)$ are purely imaginary. Hence by taking the real part,
\begin{align}
Re\ \sum_{k=1}^{N}\bar{z}_{k}\Theta_{j}(k)=0.
\end{align}

Next
we turn to $\Theta_5.$ Use the definition of $X_1$ and compute directly to obtain 
\begin{align*}
\sum_{k=1}^{N}\bar{z}_{k}\Theta_{5}(k)=-\left\langle \sum_{|m|=|n|=1}R_{m,n}z^{m}\bar{z}^{n},\ \left(
\begin{array}{lll}
i\phi^{\lambda}|z\cdot \eta|^2+3i\phi^{\lambda} |z\cdot \xi|^2\\
\phi^{\lambda}[ \bar{z}\cdot \xi \ z\cdot\eta- z\cdot \xi\cdot \bar{z}\cdot \eta]
\end{array}
\right) \right\rangle
\end{align*}
By noticing that $\displaystyle\sum_{|m|=|n|=1}R_{m,n}z^{m}\bar{z}^{n}$ is real by definition, and that the other vector function is purely imaginary, we have that the inner product is purely imaginary, hence
\begin{align}
Re\ \sum_{k=1}^{N}\bar{z}_{k}\Theta_{5}(k)=0.
\end{align}

Collect the estimates above to complete the proof.

\section{Proof of the Main Theorem \ref{THM:MainTheorem}}\label{ProveMain} 
The proof is almost identical to the part in \cite{GaWe}, hence we only sketch it.

We begin by introducing a family of space-time norms, $ Z(T),\ {\cal R}_j(T)$,  for measuring the decay of the $z(t)$ and  $\vec{R}(t)$  for $0\le t\le T$, with $T$ arbitrary and large.
We then prove that this family of norms satisfy a set of coupled inequalities,
from which we can infer the desired large time asymptotic behavior. 
Define
\begin{equation}
T_{0}:=|z(0)|^{-1},\label{Todef}
\end{equation}
where, recall that $|z(0)|$ is the initial amount of mass of neutral modes, see Theorem \ref{THM:MainTheorem}.

Now we define the controlling functions:
\begin{equation}\label{majorant}
\begin{array}{lll}
Z(T):=\displaystyle\max_{t\leq T}(T_{0}+t)^{\frac{1}{2}}|z(t)|,& &
\mathcal{R}_{1}(T):=\displaystyle\max_{t\leq
T}(T_{0}+t)\|\langle x\rangle^{-\nu} \vec{R}(t)\|_{\mathcal{H}^{3}},\\
\mathcal{R}_{2}(T):=\displaystyle\max_{t\leq
T}(T_{0}+t)\|\vec{R}(t)\|_{\infty},& &
\mathcal{R}_{3}(T):=\displaystyle\max_{t\leq
T}(T_{0}^{\frac{2}{3}}+t)^{\frac{7}{5}}\|\langle x\rangle^{-\nu}\tilde{R}(t)\|_{2},\\
\mathcal{R}_{4}(T):=\displaystyle\max_{t\leq
T}\|\vec{R}(t)\|_{\mathcal{H}^{3}},& &
\mathcal{R}_{5}(T):=\displaystyle\max_{t\leq
T}\frac{(T_{0}+t)^{\frac{1}{2}}}{\log(T_{0}+t)}\|\vec{R}(t)\|_{3}\ .\\
\end{array}
\end{equation}

To estimate $\mathcal{R}_{k}$, $k=1,2,3,5,$ or to control the dispersion $\vec{R}$, we use the propagator estimate, namely for any function $g\in \mathcal{L}^{1},$
\begin{align}
\|e^{tL(\lambda)}P_c^{\lambda}g\|_{\infty}\lesssim t^{-\frac{3}{2}}\|g\|_1.
\end{align} To estimate the decay of $|z|$, we define a positive constant $q$ by $$q^2=|z|^2+F(z,\bar{z})$$ where $F(z,\bar{z})$ is real and of order $|z|^4$, as stated in Statement C of Theorem \ref{GOLD:maintheorem}. By Statement C of Theorem \ref{GOLD:maintheorem}
\begin{align}
\frac{d}{dt}|q|^2\leq -C|q|^{4}+\cdots.
\end{align} If the omitted part is small, then we have $|q|^{2}=|z|^2\lesssim (1+t)^{-1}.$

The details are identical to the corresponding parts in \cite{GaWe}, and the proof is tedious. Hence we omit the detail.
The results are:
\begin{proposition}\label{prop:majorants}
\begin{align*}
\mathcal{R}_{1}\lesssim &T_{0}\|\langle x\rangle^{\nu}\vec{R}(0)\|_{\mathcal{H}^{2}}+\mathcal{R}_{4}\mathcal{R}_{2}+Z^{2}+T_{0}^{-\frac{1}{2}}[Z^{3}+Z\mathcal{R}_{1}+\mathcal{R}_{1}^{2}+\mathcal{R}_{2}^{2}].\\
\mathcal{R}_{2}\lesssim &T_{0}\|\vec{R}(0)\|_{1}+T_{0}\|\vec{R}(0)\|_{\mathcal{H}^{2}}+Z^{2}+\mathcal{R}_{4}^{2}\mathcal{R}_{2}+T_{0}^{-\frac{1}{2}}[Z^{3}+Z\mathcal{R}_{1}+\mathcal{R}_{1}^{2}].\\
\mathcal{R}_{3}\lesssim  &
T_{0}\|\langle x\rangle^{\nu}\vec{R}(0)\|_{2}+T_{0}^{\frac{3}{2}}|z|^{2}(0)+T_{0}^{-\frac{1}{20}}
(Z^{3}
+Z\mathcal{R}_{3}+Z\mathcal{R}_{1}+\mathcal{R}_{5}^{3}+R_{2}^{2}\mathcal{R}_{4}).\\
\mathcal{R}_{4}^{2}\lesssim  &
\|\vec{R}(0)\|^{2}_{\mathcal{H}^{2}}+T_{0}^{-1}[\mathcal{R}_{1}^{2}+Z^{2}\mathcal{R}_{1}+Z^{2}\mathcal{R}_{1}^{2}+\mathcal{R}_{4}^{2}\mathcal{R}_{2}^{2}].\\
\mathcal{R}_{5}\lesssim  & T_{0}\|\vec{R}(0)\|_{1}+T_{0}\|\vec{R}(0)\|_{\mathcal{H}^{2}}+Z^{2}+T_{0}^{-\frac{1}{2}}[\mathcal{R}_{5}^{2}\mathcal{R}_{2}+Z^{4}+Z\mathcal{R}_{3}+\mathcal{R}_{1}^{2}+\mathcal{R}_{2}^{2}].\\
Z(T) \lesssim  &1 + \frac{1}{T_0^{2\over5}}\ Z(T)\ \left(\ Z(T)+ \cR_1^2(T)+\cR_2^2(T)+Z(T)\cR_3(T)\ \right)
\end{align*}

\end{proposition}


\subsection{Proof of the Main Theorem ~\ref{THM:MainTheorem}}\label{subsec:proofsubmaintheorem}

Our goal is to prove that the functions $Z$ and $\mathcal{R}_{k},\ k=1,2,\cdots,5$, are uniformly bounded, using the estimates in Proposition \ref{prop:majorants}.

Define
$M(T):=\displaystyle\sum_{n=1}^{4}\mathcal{R}_{n}(T)$ and
\begin{equation}\label{defineS}
S:=T_{0}(\|\vec{R}(0)\|_{\mathcal{H}^{2}}+\|\langle x\rangle^{\nu}\vec{R}(0)\|_{2}+\|\vec{R}(0)\|_1)+\|\vec{R}(0)\|_{\mathcal{H}^2}+T_0^{\frac{3}{2}}|z_0|^2,
\end{equation} where, recall the definition of $T_{0}$ after
\eqref{majorant}. By the conditions on the initial condition \eqref{InitCond}
we have that $\mathcal{R}_{4}(0)$ is small, $M(0)$, $Z(0)$ and $S$ are
bounded. Then by the estimates in Proposition \ref{prop:majorants} we
obtain 
\begin{align}
M(T)+Z(T)\leq \mu(S), \text{and}\ \mathcal{R}_{4}\ll 1,
\end{align} where
$\mu$ is a bounded function if $S$ bounded. This together with the definitions of $\mathcal{R}_k,\ k=1,2,\cdots,5,$ implies that
\begin{align}
\|\langle x\rangle^{-\nu}\vec{R}\|_{2},\ \|\vec{R}\|_{\infty}\leq
c(T_{0}+t)^{-1},\ |z(t)|\leq c(T_{0}+t)^{-\frac{1}{2}}\label{equa619a}
\end{align} which is part of Statements (A) and (C) in Theorem
~\ref{THM:MainTheorem}. The rest of (A), Equation
\eqref{eq:detailedDescription1}, is proved in
\eqref{eq:detailedDescription}.

\begin{flushright}
$\square$
\end{flushright}

\end{document}